\documentclass[11pt,a4paper,reqno]{amsart}
\usepackage[dvipsnames]{xcolor}
\allowdisplaybreaks

\newcommand{\zero}[1]{{#1}_{\scriptscriptstyle{(0)}}}
\newcommand{\one}[1]{{#1}_{\scriptscriptstyle{(1)}}}
\newcommand{\two}[1]{{#1}_{\scriptscriptstyle{(2)}}}

\newcommand{\mone}[1]{{#1}_{\scriptscriptstyle{(-1)}}}

\newcommand{\ot}{\otimes}
\newcommand{\bbC}{\mathbb{C}}

\newcommand{\id}{\mathrm{id} }
\newcommand{\ci}{\sqrt{-1} }
\newcommand{\del}[1]{{}^{#1}\delta}

\renewcommand{\cot}{\gamma}
\newcommand{\coin}[2]{\bar{\cot}\big({#1}\otimes{#2}\big)}
\newcommand{\co}[2]{\cot\big({#1}\ot{#2}\big)}
\newcommand{\cat}{\qMod{A}{B}{}{B}}
\newcommand{\tcat}{\qMod{A_\gamma}{B_\gamma}{}{B_\gamma}}
\newcommand{\catmodbb}{\qMod{}{B}{}{B}}
\newcommand{\cateibb}{{}_B{\mathcal{E}\mathcal{I}}_B}
\newcommand{\hm}[3]{{}_{#1}\mathrm{Hom}(#2, #3)}

\newcommand{\ev}{\mathrm{ev}}
\newcommand{\coev}{\mathrm{coev}}

\newcommand{\dc}{(\Omega^\bullet, \wedge, d)}
\newcommand{\dctwisted}{(\Omega_\gamma^{\bullet}, \wedge_\gamma, d_\gamma)}

\newcommand{\metric}{(g, ( ~, ~ ))}
\newcommand{\metrictwisted}{(g_\gamma,(~,~)_\gamma)}

\newcommand{\braces}[1]{$\mathrm{(}$#1$\mathrm{)}$}

\newcommand{\ol}[1]{\overline{#1}}

\newcommand{\form}[1]{\Omega^{#1}}
\newcommand{\uform}[1]{\Omega^{#1}_u}

\newcommand{\Title}{On two notions of torsion and metric compatibility of connections in noncommutative geometry}%
\newcommand{\ShortTitle}{\Title}%

\newcommand{\AuthorOne}{Bappa Ghosh}%
\newcommand{\AuthorTwo}{Jyotishman Bhowmick}
\newcommand{\AuthorThree}{Satyajit Guin}
\newcommand{\AuthorOneAddr}{%
Stat-Math Unit, Indian Statistical Institute, Kolkata, India 700108
}%
\newcommand{\AuthorThreeAddr}{Department of Mathematics and Statistics,
Indian Institute of Technology Kanpur,
Uttar Pradesh, India 208016}
\newcommand{\AuthorOneEmail}{%
bappa0697@gmail.com}
\newcommand{\AuthorTwoEmail}{jyotishmanbmath@gmail.com}

\newcommand{\AuthorThreeEmail}{sguin@iitk.ac.in}%
\newcommand{\AuthorOneThanks}{%
BG is supported by CSIR, India through the SPM Fellowship.
}%

\newcommand{\Keywords}{Mathematics, Your Specific Keyword, Another One}

\newcommand{\pdfTitle}{PDF Title}
\newcommand{\pdfAuthor}{PDF Author(s)}
\newcommand{\pdfSubject}{PDF Subject}
\newcommand{\pdfKeywords}{\Keywords}
\newcommand{\pdfCreator}{Write a creator name (e.g. MikTeX)}
\newcommand{\pdfCreationDate}{\today}
\newcommand{\pdfColorLink}{true}
\newcommand{\pdfLinkColor}{blue}
\newcommand{\pdfUrlColor}{blue}
\newcommand{\pdfCiteColor}{blue}

\def\clap#1{\hbox to 0pt{\hss#1\hss}}
\def\mathllap{\mathpalette\mathllapinternal}

\def\mathllapinternal#1#2{%
	\llap{$\mathsurround=0pt#1{#2}$}}

\newsavebox\qModFirst
\newsavebox\qModSecond
\newcommand{\Mod}{\mathsf{Mod}}
\newcommand{\modz}[2]{
	\sbox{\qModFirst}{$#1$}
	\sbox{\qModSecond}{$#2$}
	\ifdim\wd\qModFirst>\wd\qModSecond
	{}^{\phantom{#1}\mathllap{#1}}_{\phantom{#1}\mathllap{#2}}%
	\else
	{}^{\phantom{#2}\mathllap{#1}}_{\phantom{#2}\mathllap{#2}}%
	\fi\mathsf{mod}_0}
\newcommand{\Modz}[2]{
	\sbox{\qModFirst}{$#1$}
	\sbox{\qModSecond}{$#2$}
	\ifdim\wd\qModFirst>\wd\qModSecond
	{}^{\phantom{#1}\mathllap{#1}}_{\phantom{#1}\mathllap{#2}}%
	\else
	{}^{\phantom{#2}\mathllap{#1}}_{\phantom{#2}\mathllap{#2}}%
	\fi\mathsf{Mod}_0}
\newcommand{\qMod}[4]{{%
		\sbox{\qModFirst}{$#1$}
		\sbox{\qModSecond}{$#2$}
		\ifdim\wd\qModFirst>\wd\qModSecond
		{}^{\phantom{#1}\mathllap{#1}}_{\phantom{#1}\mathllap{#2}}%
		\else
		{}^{\phantom{#2}\mathllap{#1}}_{\phantom{#2}\mathllap{#2}}%
		\fi
		\Mod^{#3}_{#4}}}

\newcommand{\qmod}[4]{{%
		\sbox{\qModFirst}{$#1$}
		\sbox{\qModSecond}{$#2$}
		\ifdim\wd\qModFirst>\wd\qModSecond
		{}^{\phantom{#1}\mathllap{#1}}_{\phantom{#1}\mathllap{#2}}%
		\else
		{}^{\phantom{#2}\mathllap{#1}}_{\phantom{#2}\mathllap{#2}}%
		\fi
		\mathsf{mod}^{#3}_{#4}}}
\newcommand{\lmod}[2]{\qmod{#1}{#2}{}{}}

\usepackage[top=0.9in, bottom=0.89in, left=0.9in, right=0.9in]{geometry}
\usepackage{amsmath, amssymb, amsthm} 
\usepackage[utf8]{inputenc}
\usepackage[T1]{fontenc}
\usepackage{mathtools}
\usepackage{mathrsfs} 
\usepackage{xfrac} 
\usepackage{array}
\usepackage{verbatim}
\usepackage{graphicx}
\usepackage{enumitem} 
\usepackage{hyperref}
\hypersetup{
	pdftitle={\pdfTitle},
	pdfauthor={\pdfAuthor},
	pdfsubject={\pdfSubject},
	pdfcreationdate={\pdfCreationDate},
	pdfcreator={\pdfCreator},
	pdfkeywords={\pdfKeywords},
	colorlinks=\pdfColorLink,
	linkcolor={\pdfLinkColor},
	urlcolor=\pdfUrlColor,
	citecolor=\pdfCiteColor,
	pdfpagemode=UseOutlines,
}
\usepackage{tikz-cd} 
\usepackage{lipsum}
\usetikzlibrary{matrix,arrows}
\usepackage[english]{babel}
\usepackage{lmodern}
\usepackage{xcolor}

\theoremstyle{plain}
\newtheorem{theorem}{Theorem}[section]
\newtheorem{prop}[theorem]{Proposition}
\newtheorem{lem}[theorem]{Lemma}
\newtheorem{cor}[theorem]{Corollary}

\theoremstyle{plain}
\newtheorem{definition}[theorem]{Definition}
\newtheorem{example}[theorem]{Example}

\theoremstyle{plain}
\newtheorem{remark}[theorem]{Remark}

\numberwithin{equation}{section}

\newtheoremstyle{ser}
{8pt}
{8pt}
{\it}
{}
{\sf}
{:}
{6mm}
{}

\theoremstyle{ser}

\newtheoremstyle{serr}
{8pt}
{8pt}
{\normalfont}
{}
{\sf}
{.}
{6mm}
{}

\theoremstyle{serr}

\theoremstyle{ser}

\theoremstyle{ser}

\definecolor{indraRed}{rgb}{0.593, 0.183, 0.183}
\definecolor{indraPink}{rgb}{0.858, 0.188, 0.478}
\definecolor{indraBlue}{rgb}{0, 0.199, 0.398}
\definecolor{madridBlue}{rgb}{0.199, 0.199, 0.695}
\definecolor{metropolisThemeColor}{rgb}{0.105, 0.214, 0.234}
\definecolor{metropolisBarColor}{rgb}{0.984, 0.0.515, 0.015}
\definecolor{UBCblue}{rgb}{0.04706, 0.13725, 0.26667} 
\definecolor{UBCgrey}{rgb}{0.3686, 0.5255, 0.6235} 

\makeatletter
\def\mathcolor#1#{\@mathcolor{#1}}
\def\@mathcolor#1#2#3{%
	\protect\leavevmode
	\begingroup
	\color#1{#2}#3%
	\endgroup
}
\makeatother

\newcommand{\ra}{\rightarrow} 

\newcommand{\Hom}{{\rm{Hom}}}

     \newcommand{\cE}{\mathcal E}		
     \newcommand{\cF}{\mathcal F}		
     		
     		\newcommand{\kH}{\mathscr{H}}

\newcommand{\fN}{\mathfrak N}     		
     \newcommand{\cO}{\mathcal O}

\newcommand{\bN}{{\mathbb N}}

\begin{document}

\title[\ShortTitle]{\MakeUppercase\Title}
\author{\AuthorTwo}
\author{\AuthorOne}
\author{\AuthorThree} {\thanks{\AuthorOneThanks}}


\address{\AuthorOneAddr}
\email{\AuthorTwoEmail}
\address{\AuthorOneAddr}
\email{\AuthorOneEmail}
\address{\AuthorThreeAddr}
\email{\AuthorThreeEmail}



\begin{abstract}
  We compare the notions of metric-compatibility and torsion of a connection in the frameworks of Beggs-Majid and Mesland-Rennie. It follows that for $\ast$-preserving connections, compatibility with a real metric in the sense of Beggs-Majid corresponds to Hermitian connections in the sense of Mesland-Rennie. If the calculus is quasi-tame, the torsion zero conditions are  equivalent. A combination of these results proves the existence and uniqueness of Levi-Civita connections in the sense of Mesland-Rennie for unitary cocycle deformations of a large class of Riemannian manifolds as well as the Heckenberger-Kolb calculi on all quantized irreducible flag manifolds.    \end{abstract}
\maketitle%
\thispagestyle{empty}%
\section{Introduction}

The first attempts to understand torsion and metric compatibility of connections in the context of noncommutative geometry dates back to the 1990 s. While the paper \cite{frohlich99} studied the relevant concepts in the context of spectral triples (\cite{connes}), the authors of \cite{heckenbergerlc} obtained existence of Levi-Civita connections for covariant calculi on certain quantum groups. Beggs and Majid (\cite{BeggsMajid:Leabh}) developed a substantial amount of machinery for studying torsion and metric compatibility of bimodule connections and also proved the existence and uniqueness of Chern connections on noncommutative holomorphic bimodules. Following the work of Rosenberg (\cite{rosenberg}), the topic received renewed attention as different versions of torsion and metric compatibility were proposed and studied in the literature. We refer to \cite{peterkasheu,arnwil2,arnwil,article1,article2,article3} and references therein. Subsequently, existence and uniqueness results for several classes of quantum groups and their homogeneous spaces were obtained in \cite{aschieriweber,atiyahseq,JBSugato,matassalevicivita,LeviCivitaHK}. 

Recently, the authors of \cite{meslandrennie1} introduced the framework of Hermitian differential structures and proved a necessary and sufficient condition for the existence of Levi-Civita connections in their sense. They proved (\cite{meslandrennie1,meslandrennie3}) that the toric deformations of classical spectral triples as well as the Dabrowski-Sitarz spectral triple on the standard Podle\'{s} sphere satisfy the conditions of their theorem. 

A recent comparative study \cite{flamantMR} investigated how the approach taken by Mesland and Rennie (\cite{meslandrennie1}) relates to the  vector‑field approach of \cite{arnwil2, arnwil} as well as a different approach taken in \cite{article1,article2,article3}. In this article, we compare the methods proposed in \cite{BeggsMajid:Leabh} and \cite{meslandrennie1}. Thus, our objective is to clarify the relationships between the two definitions of metric compatibility as well as  torsion of connections. Then we combine our results to prove the existence and uniqueness of Levi-Civita connections in the sense of \cite{meslandrennie1} for real covariant metrics on the space of one-forms of unitary cocycle deformations of a large class of Riemannian manifolds and the Heckenberger-Kolb calculi on the quantized irreducible flag manifolds.

In Section \ref{20thaugust251}, we recall the framework of Beggs and Majid as developed in \cite{BeggsMajid:Leabh}. The starting point is the datum of a differential calculus $\dc$ over an algebra $B.$ The torsion of a (left) connection $\nabla: \Omega^1 (B) \rightarrow \Omega^1(B) \otimes_B \Omega^1(B)  $ is the left $B$-linear map $T_\nabla:= \wedge \circ \nabla - d: \Omega^1(B) \rightarrow \Omega^2(B).$ Note that the definition of torsion requires the data of both the one-forms and the two-forms. On the other hand, the metric compatibility of $\nabla$ makes sense for any first order differential calculus. However, the definition needs a bit more explanation since the compatibility can be defined in terms of a metric or an Hermitian metric. 

A metric on $\dc$ is a choice of a self-dual structure $\metric$ on the space of one-forms $\Omega^1$ of the calculus. Thus, $g$ is an element of $\Omega^1(B) \otimes_B \Omega^1(B) $ and $(~, ~): \Omega^1(B) \otimes_B \Omega^1(B) \rightarrow \Omega^1(B) $ is a $B$-linear map satisfying some compatibility conditions. If $(\nabla_{\Omega^1 (B)}, \sigma)$ is a bimodule connection (see Subsection \ref{4thaugust241}) on $\Omega^1(B),$ then $\nabla_{\Omega^1 (B)}$ extends to a connection $\nabla_{\Omega^1 (B) \otimes_B \Omega^1(B)}$ on $\Omega^1(B) \otimes_B \Omega^1(B).$ The authors of \cite{BeggsMajid:Leabh} call $\nabla_{\Omega^1(B)}$ to be compatible with a metric $\metric$ as above if $\nabla_{\Omega^1 (B) \otimes_B \Omega^1(B)} g = 0.$  When the algebra $B$ and the calculus $\dc$ have compatible $\ast$-structures, Beggs and Majid introduced a notion of Hermitian metrics given by certain $B$-bimodule isomorphisms $\kH: \overline{\Omega^1 (B) } \rightarrow \hm{B}{\Omega^1(B)}{B} $ and there is a notion of compatibility of connections (not necessarily bimodule connections) with $\kH$ (see Definition \ref{19thoct231}). If $\metric$ is a real metric (Definition \ref{23rddec24jb2}) on the space of one-forms of a $\ast$-differential calculus, then one can construct an Hermitian metric $\kH_g$ on $\Omega^1(B).$ Thus, it is a natural question to ask whether the compatibility of a connection $\nabla$ on $\Omega^1(B)$
with $\metric$ is related to the compatibility of $\nabla$ with $\kH_g.$ The result Proposition \ref{21stapril251} which was already observed in \cite{BeggsMajid:Leabh} states that these two notions coincide when the connection $\nabla$ is $\ast$-preserving.   This is the key result  which allows us to compare the frameworks of \cite{BeggsMajid:Leabh} and \cite{meslandrennie1}.

Mesland and Rennie's notion of compatibility is based on the framework of $\dagger$-bimodules. For an algebra $B,$  these are $B$-bimodules $\cE$ equipped with an antilinear involution $\dagger$ and a positive definite sesquilinear form  ${}_B\langle \langle ~ , ~ \rangle \rangle : \cE \times \cE \rightarrow B $ satisfying the condition ${}_B\langle \langle e b , f \rangle \rangle = {}_B\langle \langle e , f b^* \rangle \rangle $  for all $e, f \in \cE, b \in B.$ Connections on $\cE$ which are compatible with $ {}_B\langle \langle ~ ,  ~  \rangle \rangle $ are called Hermitian connections. In this article, we drop the positivity condition on $ {}_B\langle \langle ~ ,  ~  \rangle \rangle $ and work with a larger class of sesquilinear forms which we call weak $\dagger$-bimodules. If $(\Omega^1(B), d)$ is a first order $\ast$-differential calculus, then it is easy to see that weak $\dagger$-bimodules $(\Omega^1(B), \dagger,  {}_B\langle \langle ~ ,  ~  \rangle \rangle  )$  are closely related to Hermtian metrics on $\Omega^1(B)$ in the sense of \cite{BeggsMajid:Leabh}. However, our main goal in Section \ref{21staugust251} is to relate compatibility of connections with metrics $\metric$ on $\Omega^1(B)$ with Hermitian connections with respect to ${}_B\langle \langle ~ , ~ \rangle \rangle.$ This is where we assume $\metric$ to be real and use Proposition \ref{21stapril251}. Our main result in Section \ref{21staugust251} is Theorem \ref{9thmay252} where we prove that connections on $\Omega^1(B)$ compatible with real metrics $\metric$ are in one to one correspondence with connections which are Hermitian with respect to a canonical weak $\dagger$-bimodule structure on $\Omega^1(B)$ associated to $\metric.$

In Section \ref{19thjune251}, we compare the definitions of torsions of connections studied in \cite{BeggsMajid:Leabh} and \cite{meslandrennie1}. The authors of the latter paper consider second order differential structures given by a first order calculus $(\Omega^1, d)$ over an algebra $B$ equipped with a $\dagger$-structure and  postulate the existence of a $\dagger$-compatible $B$-bilinear idempotent 
$\psi: \Omega^1(B) \otimes_B \Omega^1(B) \rightarrow \Omega^1(B) \otimes_B \Omega^1(B) $ whose range contains certain junk $2$-tensors. This allows them to use the universal differential $ \Omega^1_u (B) \rightarrow \Omega^2_u (B) $ to define a map $d_{\psi}: \Omega^1(B) \rightarrow \Omega^1(B) \otimes_{B} \Omega^1(B) $ and a connection $\nabla$ on $\Omega^1(B)$ is called torsionless if
\begin{equation} \label{21staugust253}
(1 - \psi) \nabla = d_{\psi}.     
\end{equation}
In the first two subsections of Section \ref{19thjune251} where we discuss torsion, we have avoided the usage of any $\dagger$-structure on the calculus and thus, we have not assumed any compatibility between $\psi$ and any $\dagger$-structure. Thus, we have considered  the class of weak second order differential structures $(\Omega^1(B), d, \psi)$ such that $\Omega^1(B)$ is a sub-bimodule of an algebra $M.$ We clarify the relationship between the conditions \eqref{21staugust253} and the Beggs-Majid equation for zero-torsion for the Connes' calculus associated to the inclusion $\Omega^1(B) \hookrightarrow M. $ Subsequently, we focus our attention on tame differential calculi $\dc$ for which there exists a $B$-bilinear map $s: \Omega^2(B) \rightarrow \Omega^1(B) \otimes_B \Omega^1(B) $ satisfying the condition $\wedge \circ s = \text{id}_{\Omega^2(B)}. $ Such calculi were already studied in the literature for various purposes (see \cite{BeggsMajid:Leabh,article1,article3}). In Theorem \ref{14thjan251}, we prove that if $(\Omega^\bullet(B), \wedge, d, s)$ is a tame differential calculus, then  the Beggs-Majid definition of a torsion zero connection on $\Omega^1(B)$ coincides with \eqref{21staugust253} for a weak second order differential structure associated to the tame differential calculus. 

Thus, a combination of Theorem \ref{9thmay252} and theorem \ref{14thjan251} provides a passage between the frameworks of Beggs-Majid and Mesland-Rennie. This has been noted in Theorem \ref{thm:levicomp}. Modulo technicalities, this theorem implies that if we have a quasi-tame $\ast$-differential calculi on a $\ast$-algebra $B$ and a $\ast$-preserving unique Levi-Civita connection (in the sense of \cite{BeggsMajid:Leabh}) for a real metric $\metric$ on $\Omega^1(B),$ then there exists a unique Hermitian connection on $\Omega^1(B)$ satisfying \eqref{21staugust253}. This enables us to apply this result to two classes of Hopf-algebra covariant calculi, namely unitary cocycle deformations of a canonical  differential calculi on the set of real points $X (\mathbb{R}) $ of a smooth real affine variety $X$ and the Heckenberger-Kolb calculi. The proof of existence and uniqueness of Levi-Civita connections (in the sense of \cite{BeggsMajid:Leabh}) on these calculi for real covariant metrics were already derived in \cite{BeggsMajid:Leabh} and \cite{LeviCivitaHK}. Thus, we need to verify the hypotheses of Theorem \ref{thm:levicomp} for these examples. This is done in Section \ref{21staugust254}.

\vspace{0.15 cm}

\paragraph{\textbf{Acknowledgments:}}
We would like to thank Bram Mesland and Adam Rennie for some clarifications regarding \cite{meslandrennie1} and \cite{flamantMR}.  We are grateful to  R\'eamonn  \'O{} Buachalla for sharing the results of the preprint \cite{ReAleJun} and a careful reading of the manuscript. B.G. gratefully acknowledges Suvrajit Bhattacharjee for many insightful discussions and for pointing out the result \cite[Proposition 4.1]{junaidbuachalla1}.

\section{Preliminaries} \label{20thaugust251}
Throughout this article, unless other mentioned, our ground field will be $\mathbb{C}.$ All algebras are assumed to be unital. All unadorned tensor products are over $\mathbb{C}.$ 

Let $ (\mathcal{C}, \otimes) $ be an monoidal category with unit object $1_{\mathcal{C}}$. An object $M$ in the category ${\mathcal{C}}$ is said to have a~right dual if there exists an object $\prescript{\ast}{}{M}$ in ${\mathcal{C}}$ and  morphisms $\ev: M\otimes \prescript{\ast}{}{M} \to 1_{\mathcal{C}}$, $\coev: 1_{\mathcal{C}}\to \prescript{*}{}{M}\otimes M$ such that the equations
  \begin{align*}
      (\ev\otimes \id_M) (\id_M\otimes  \coev) =\id_M, \quad (\id_{\prescript{\ast}{}{M}}\otimes\ev)(\coev\otimes \id_{\prescript{\ast}{}{M}})  = \id_{\prescript{\ast}{}{M}}
  \end{align*}
are satisfied.  The left dual is defined  analogously (see \cite[Definition 2.10.1]{etingof2015tensor}).

For an algebra $B,$ the notation $\catmodbb$ will stand for the monoidal category of $B$-bimodules. 
Let $\cE, \cF$ be two objects in the category $\catmodbb$. Then,
\begin{align*}
	\hm{B}{\cE}{\cF}:=  \{f: \cE \to \cF : f \text{ is  left  $B$-linear}\}
\end{align*} 
is a   $B$-bimodule with the operations
\begin{align} \label{19thdec234}
	(b\cdot f)(e)=f(e\cdot b) \quad (f \cdot b)(e) = f(e)b,
\end{align}
where $f \in \prescript{}{B}{\Hom(\cE,B)}, e\in \cE $ and $ b \in B$.

Let us recall the definition of bar categories from \cite{BeggsMajid:Leabh}. For a monoidal category $ (\mathcal{C}, \otimes)$ and objects $X, Y $ of $ \mathcal{C}, $ we denote by flip the functor from  $\mathcal{C}\times \mathcal{C} $ to $\mathcal{C}  \times \mathcal{C}$ that sends the pair $(X,Y)$ to $ (Y,X)$. As usual, we will suppress the notations for the left unit, right unit as well as the associator of $\mathcal{C}$.

\begin{definition}$\mathrm{(}$\cite[Definition 2.99]{BeggsMajid:Leabh}$\mathrm{)}$ \label{15thjuly241}
	A bar category is a monoidal category $(\mathcal{C},\otimes,1_{\mathcal{C}})$ together with  a functor $\mathrm{bar}: \mathcal{C}\to \mathcal{C}$ (written as $X\mapsto \overline{X}$),
	a natural equivalence $\mathrm{bb}: \id_ {\mathcal{C}} \to \mathrm{bar}\circ \mathrm{bar}  $ between the identity and the $\mathrm{bar}\circ \mathrm{bar}$ functors on $\mathcal{C}$,
	an invertible morphism $ \star:1_{\mathcal{C}} \to \overline{1_{\mathcal{C}}},$ and
	a natural equivalence $\Upsilon$  between $\mathrm{bar}\circ \otimes$ and $\otimes \circ (\mathrm{bar}\times \mathrm{bar})\circ \mathrm{flip}$ from $\mathcal{C}\times \mathcal{C}$ to $\mathcal{C}$	
	such that the following compositions of morphisms are both equal to $1_{\overline{X}}:$
	$$ \overline{X} \xrightarrow{\cong} \overline{X \otimes 1_{\mathcal{C}}} \xrightarrow{\Upsilon_{X, 1_{\mathcal{C}}}} \overline{ 1_{\mathcal{C}}}\otimes  \overline{X} \xrightarrow{\star^{-1} \otimes \id}  1_{\mathcal{C}} \otimes \overline{X} \xrightarrow{\cong} \overline{X}, $$
	$$ \overline{X} \xrightarrow{\cong} \overline{ 1_{\mathcal{C}} \otimes X} \xrightarrow{\Upsilon_{ 1_{\mathcal{C}},X}}  \overline{X} \otimes\overline{1_{\mathcal{C}}}\xrightarrow{ \id \otimes  \star^{-1}}  \overline{X} \otimes  1_{\mathcal{C}} \xrightarrow{\cong} \overline{X}.$$ 
	Moreover, the following equations hold:
	$$ (\Upsilon_{Y,Z} \otimes \id) \Upsilon_{X, Y \otimes Z} = (\id \otimes \Upsilon_{X, Y}) \Upsilon_{X \otimes Y, Z}, ~ \overline{\star} \star = \mathrm{bb}_{1_{\mathcal{C}}}: 1_{\mathcal{C}} \to \overline{\overline{1_{\mathcal{C}}}}, ~ \overline{\mathrm{bb}_X} =\mathrm{bb}_{\overline{X}}: \overline{X} \to \overline{\overline{\overline{X}}}$$
	for all objects $X, Y, Z$ in $\mathcal{C}$.
	
	An object $X$ in a bar category is  called a star object if there is a morphism $\star_X: X\to \overline{X}$ such that 
    \begin{equation} \label{10thmay253}
    \overline{\star_X}\circ \star_X = \mathrm{bb}_X .
    \end{equation}
\end{definition}

In this article, for a vector space $M$, the symbol $\overline{M}$ will denote the conjugate vector space defined as $ \overline{M}:= \{ \overline{m}: m \in M \}.$ 
Moreover, if $M$ is $B$-bimodule, then $\overline{M}$ is equipped with the following $B$-bimodule structure:
\begin{align}\label{28thnov231}
	b\cdot \ol{m}= \ol{m\cdot b^{\ast}};\quad \ol{m}\cdot b= \ol{b^{\ast}\cdot m} \quad \text{for all $b\in B$, $m\in M$}.
\end{align}

\begin{example} $\mathrm{(}$\cite[Section 2.8]{BeggsMajid:Leabh}$\mathrm{)}$ \label{5thdec241jb}
 If $B$ is an algebra, then $\catmodbb$ is a bar category.	Indeed, if $M$ and $N$ are  objects in $\catmodbb,$ then define    $\mathrm{bar} (M) := \overline{M},$ and 
	for $f \in \Hom(M,N)$,  define 
	$ \ol{f}\in \Hom(\ol{M}, \ol{N})$ by $\ol{f}(\ol{x})=\ol{f(x)}$.	
	The natural equivalence $\mathrm{bb}$ is given by
	$ \mathrm{bb}_M ( m ) = \overline{\overline{m}} $ for all $m \in M.$
	Finally, for objects $M, N$ in $\catmodbb,$
	$ \Upsilon_{M, N} (  \overline{m \otimes_B n}  ) = \overline{n} \otimes_B \overline{m}.  $	
	\end{example}

\subsection{Differential calculi} \label{7thmay251}

\begin{definition} $\mathrm{(}$\cite[Definition 1.1]{woronowicz1989differential}$\mathrm{)}$ \label{5thmay251}
A first order differential  calculus on a unital algebra $B$ is a pair $(\Omega^1 ( B ), d), $ where $\Omega^1 ( B ) $ is a $B$-bimodule and $d: B \rightarrow \Omega^1 ( B ) $ is a derivation such that      $ \Omega^1 ( B ) = {\text span} \{ b dc: b, c \in B \}. $
\end{definition}

We will refer to $\Omega^1 ( B ) $ as the bimodule of one-forms for the first order differential calculus. The bimodule $\Omega^1 ( B ) $ can be realized as the quotient of the universal space of one-forms for $B.$
More concretely, if $m: B \otimes B \rightarrow B $ denotes the multiplication map, then the universal space of one-forms is defined as
$ \Omega^1_u ( B ) = {\rm Ker} ( m ) $ and we have a derivation
$$ \delta: B \rightarrow \Omega^1_u ( B ) ~ {\rm defined} ~ {\rm by} ~ \delta ( b ) =  1 \otimes b - b \otimes 1. $$
We recall  that if $B$ is a unital algebra and $M$ is a $B$-bimodule such that there exists a derivation $d: B \rightarrow M,$ then there exists a $B$-bimodule map $\pi: \Omega^1_u ( B ) \rightarrow M $ such that 
\begin{equation*} \label{9thjan254}
d = \pi \circ \delta.
\end{equation*}


\begin{definition}
A differential calculus on an algebra $B$ is a differential graded algebra $(\Omega^{\bullet}(B) = \bigoplus_{k\ge 0} \Omega^k(B),  \wedge, d)$ such that $\Omega^0(B) = B$ and $ \Omega^\bullet (B) $ is generated as an algebra by $B$ and $dB$. 
\end{definition}

\begin{example} \label{6thmay252}
For an algebra $B$ and $\Omega^1_u ( B ) $ as above, define $\Omega^0_u ( B ) = B, $
$$ \Omega^k_u ( B ) := \underbrace{\Omega^1( B ) \otimes_B \cdots \otimes_B \Omega^1 ( B )}_{k  \text{-times}}, ~ \Omega^\bullet_u ( B ) = \oplus_{k \geq 0} \Omega^k_u ( B ). $$
Let
$$ \delta: \Omega^k_u ( B ) \rightarrow \Omega^{k + 1}_u ( B ) ~ \text{ be defined as}  ~ \delta ( b_0 \delta ( b_1 ) \cdots \delta ( b_k ) ) = \delta ( b_0 ) \delta ( b_1 ) \cdots \delta ( b_k )   $$
 and $\wedge: \Omega^k_u ( B ) \otimes_B \Omega^l ( B ) \rightarrow \Omega^{k + l} ( B )  $ be the obvious map. Then, $ ( \Omega^\bullet_u ( B ), \wedge, d )  $ is a differential calculus, known as the universal differential calculus on $B.$
\end{example}

From now on, while referring to a differential calculus on $B$, we will often use the notations $\Omega^k$ and $\Omega^{\bullet}$ to denote $\Omega^k(B)$ and $\Omega^{\bullet}(B)$ respectively.

If $\dc$ is a differential calculus on $B,$ then a graded differential ideal of $\Omega^\bullet$ is a two-sided ideal $I = \oplus_{k \geq 0} I_k, $ where $I_k$ is a two-sided ideal of $\Omega^k$ for all $k \geq 0,$ and  $d ( I_k) \subseteq I_{k+1}. $
In this case, one has a quotient differential calculus $( \widetilde{\Omega^\bullet}, \widetilde{\wedge}, \widetilde{d} )$ where $ \widetilde{\Omega^\bullet} = \oplus_{k \geq 0} \widetilde{\Omega^k},$
\begin{equation*} 
 \widetilde{\Omega^k} = \Omega^k/I_k, \widetilde{\wedge} ( ( \omega + I_k  ) \otimes_B ( \eta + I_l  )  ):= \omega \wedge \eta + I_{k + l}, ~ \widetilde{d} ( \omega + I_k  ) = d \omega + I_{k + 1}
 \end{equation*}
for all $\omega \in \Omega^k, \eta \in \Omega^l.$




We end this subsection with the definition of $\ast$-differential calculi.

\begin{definition} \label{9thmay253}
Let $B$ be a $\ast$-algebra. A first order differential calculus $( \Omega^1, d )$ on $B$ is called a first order $\ast$-differential calculus if there exists an antilinear involution $\ast: \Omega^1 \rightarrow \Omega^1 $ such that for all $b, c \in B$ and for all $\omega \in \Omega^1,$
$$ ( b \omega c )^\ast = c^* \omega^* b^* ~ \text{and} ~ ~ d ( b^* ) = d ( b )^*. $$
More generally, a differential calculus $(\Omega^{\bullet}, \wedge, d) $ on a $\ast$-algebra  $B$ is called a $\ast$-differential calculus if there exists a conjugate linear involution $\ast: \Omega^{\bullet} \to \Omega^{\bullet} $  extending the map $\ast: B \rightarrow B $ such that $\ast(\Omega^k)\subseteq \Omega^k$ and moreover,
\begin{equation}
    \label{eq:starcal}  (d\omega)^{\ast} = d(\omega^{\ast}), ~ (\omega \wedge \nu)^{\ast} = (-1)^{kl} \nu^{\ast} \wedge \omega^{\ast}
\end{equation}
for  all $ \omega \in \Omega^{k}, \nu \in \Omega^l $.
\end{definition}

\begin{remark} \label{12thmay252}
If $(\Omega^1, d )$ is a first order $\ast$-differential calculus, consider the $B$-bimodule map
\begin{equation} \label{14thmay251}
 \star_{\Omega^1}: \Omega^1 \rightarrow \overline{\Omega^1} ~ \text{defined by} ~ \star_{\Omega^1} ( \omega ) = \overline{\omega^\ast}.
 \end{equation}
It follows that $(\Omega^1, \star_{\Omega^1})$ is a star object (see Definition \ref{15thjuly241}) of the bar category $\catmodbb.$ 
\end{remark}

\subsection{Connections} \label{4thaugust241}

If  $(\Omega^{\bullet}, \wedge, d) $ is a differential calculus on an algebra $B$, then a left connection on a $B$-bimodule $\mathcal{E}$ is a $\mathbb{C}$-linear map
$\nabla: \mathcal{E} \rightarrow \Omega^1 \otimes_B \mathcal{E}$ such that
$$\nabla (b e) = b \nabla (e)  + db \otimes_B e $$
for all $e \in \mathcal{E}$ and for all $b \in B$. A right connection on $\cE$ is a $\mathbb{C}$-linear map $\nabla: \cE \rightarrow \cE \otimes_B \Omega^1 $ such that $\nabla ( e b ) = \nabla ( e ) b + e \otimes_B db. $

Throughout this article, unless otherwise stated, we will use the word connection to mean a left connection.

\begin{definition} \label{9thjan252}
If $\nabla$ is a left connection on $\Omega^1,$ then the torsion of $\nabla,$ denoted by $T_\nabla$ is defined as
$ T_\nabla: = \wedge \circ \nabla - d. $
The connection $\nabla$ is said to be torsionless if $T_\nabla = 0.$
\end{definition}
Note that, if $\nabla$ is a connection, then $T_\nabla$ is a left $B$-linear map. We should also mention that for connections on the space of one-forms on a spectral triple (\cite{connes}), a spectral definition of torsion has been studied, for which we refer to \cite{dabrowskisitarzadv}.

A left connection $\nabla$ on a $B$-bimodule $\mathcal{E}$ is called a left $\sigma$-bimodule connection if there exists a $B$-bimodule map 
$\sigma: \mathcal{E} \otimes_B \Omega^1 \rightarrow \Omega^1 \otimes_B \mathcal{E}$
such that
\begin{equation*}
	\nabla (e b) = \nabla (e) b + \sigma (e \otimes_B db)
\end{equation*}
for all $e \in \mathcal{E}$ and for all $b \in B$.
Similarly, a right connection $\nabla$ on $\cE$ is said to be a right bimodule connection if there exists a $B$-bimodule map $\sigma: \Omega^1 \otimes_B \cE \rightarrow \cE \otimes_B \Omega^1  $ such that
\begin{equation*} \label{3rdjune251}
 \nabla ( b e ) = b \nabla ( e ) + \sigma ( db \otimes_B e ).   
\end{equation*}
Suppose that $(\nabla_{\mathcal{E}}, \sigma_{\mathcal{E}}) $ and $ (\nabla_ {\mathcal{F}}, \sigma_ {\mathcal{F}}) $ are bimodule connections  on $B$-bimodules $\mathcal{E}$ and $\mathcal{F}$.  Then by \cite[Theorem 3.78]{BeggsMajid:Leabh},  we have a bimodule connection $ (\nabla_{\mathcal{E} \otimes_B \mathcal{F}}, \sigma_{\mathcal{E} \otimes_B \mathcal{F}}) $ on $\mathcal{E} \otimes_B \mathcal{F} $ defined as
\begin{equation} \label{4thmay241}
	\nabla_{\mathcal{E} \otimes_B \mathcal{F}}:= \nabla_{\mathcal{E}} \otimes_B \id + (\sigma_{\mathcal{E}} \otimes_B \id) (\id \otimes_B \nabla_{\mathcal{F}}), ~ \sigma_ {\mathcal{E} \otimes_B \mathcal{F}}  =  (\sigma_{\mathcal{E}} \otimes_B \id)  (\id \otimes_B \sigma_{\mathcal{F}}).
\end{equation}

\begin{definition} \braces{\cite[Definition 3.85]{BeggsMajid:Leabh}} \label{15thapril252}
 Let $(\Omega^1, d)$ be a $\ast$-differential calculus on a $\ast$-algebra $B$ and $\cE$ be  a star object in the bar category $\catmodbb.$ A bimodule connection $(\nabla, \sigma)$ on $\cE$ is said to be $\ast$-preserving if
 \begin{equation*} \label{14thmay254}
  \overline{\sigma} \Upsilon^{-1} ( \star_{\Omega^1} \otimes_B \star_{\cE} ) \nabla ( e ) = \overline{\nabla} \star_{\cE} ( e ) 
 \end{equation*}
 for all $e \in \cE,$ where $\star_{\Omega^1} ( \omega ) = \overline{\omega^\ast} $ as in \eqref{14thmay251},  $\star_{\cE}$ is the morphism in Definition \ref{15thjuly241} and $\overline{\nabla} ( \overline{x} ):= \overline{\nabla ( x )}  $ for all $x \in \cE.$
 \end{definition}
A bimodule connection $(\nabla, \sigma)$ with $\sigma$ an isomorphism is called $\ast$-compatible if
\begin{equation*} \label{15thmay251}
 \sigma = ( \star^{-1}_{\cE} \otimes_B \star^{-1}_{\Omega^1}  ) \Upsilon \overline{\sigma^{-1}} \Upsilon^{-1} (  \star_{\cE} \otimes_B \star_{\Omega^1}  ).   
\end{equation*}
In \cite[page 266]{BeggsMajid:Leabh}, the authors have observed that if $(\nabla, \sigma)$ is $\ast$-preserving, then it is $\ast$-compatible.


\subsection{Metrics on differential calculi}

\begin{definition} \braces{\cite{BeggsMajid:Leabh}} \label{4thmay242}
	A metric on $\form{1}$ is a pair $(g, (~ , ~)),$ where $g $ is an element of $\Omega^1 \otimes_B \Omega^1$ and $(~ , ~): \Omega^1 \otimes_B \Omega^1 \rightarrow B$ is a $B$-bilinear map such that for all $\omega \in \Omega^1,$
    the following equations hold:
	$$ ((\omega, ~) \otimes_B \id) g = \omega = (\id \otimes_B (~ , \omega)) g.  $$
\end{definition}    

If $\dc$ is a $*$-differential calculus, then we can make sense of the following definition:
\begin{definition}\label{23rddec24jb2}
	Suppose  that $\metric$  is a metric  on $\form{1}$ with $g = \sum_i \omega_i \otimes_B \eta_i. $
 Then $\metric$    is said to be a real metric if 
 $$g= \sum_i \eta^*_i \otimes_B \omega^*_i.  $$
\end{definition}

If $ (g, (~ , ~)) $ is a metric on $\Omega^1$, then  by \cite[Lemma 1.16]{BeggsMajid:Leabh}, $g$ is central, i.e, $b g = g b$ for all $b \in B$. This implies that the map $\coev_g$, defined by 
\begin{equation} \label{23rdmay241}
    \coev_g: B \rightarrow \Omega^1 \otimes_B \Omega^1 ~  \text{defined by} ~ ~ \coev_g (b) = b g 
\end{equation}
is $B$-bilinear. The following well-known characterization (see page 311 of \cite{BeggsMajid:Leabh}) of metrics will be used throughout the article on several occasions.

\begin{remark} \label{24thjuly241}
	If $ (g, (~ , ~)) $ is a  metric on $\Omega^1$ and $\coev_g$ is defined in \eqref{23rdmay241}, then $ (\Omega^1, (~ , ~), \coev_g) $ is both a left dual and a  right dual of $\Omega^1$ in the category $\catmodbb$. Thus, if $\Omega^1$ admits a metric, then it is finitely generated and projective both as  left and right $B$-modules. 
	
\end{remark}

We recall the definition of metric compatibility.
\begin{definition}\braces{\cite{BeggsMajid:Leabh}} \label{4thjuly241}
	Suppose that $(g, ( ~, ~ )) $ is a metric  on $\Omega^1$. A bimodule connection  $ (\nabla, \sigma) $ on $\form{1}$ is said to be compatible with $(g, (~,~))$ if $\nabla g = 0$, where, we denote  the induced connection on $\Omega^1 \otimes_B \Omega^1$ defined in \eqref{4thmay241} by the same notation $\nabla.$ 

    A left bimodule connection $(\nabla, \sigma) $ is said to be a Levi-Civita connection for $\metric$ if $\nabla$ is compatible with $\metric$ and torsionless in the sense of Definition \ref{9thjan252}. 
\end{definition}

\subsection{Hermitian metrics on differential calculi}

In this subsection, we will need the language of bar-categories introduced above. In particular, we recall that if $B$ is a $*$-algebra, then $\catmodbb$ is a bar category (see Example \ref{5thdec241jb}).
%

%

\begin{definition} \label{8thmay252}
 An Hermitian metric on a $B$-bimodule  $\cE$ is an isomorphism $ \kH: \overline{\cE} \to \hm{B}{\cE}{B}$  in the category $\catmodbb$ such that 
	\begin{equation}\label{15thdec246}
		\langle y,\overline{x}\rangle^*=\langle x, \overline{y}\rangle \quad \text{for all } y, x\in \cE,
	\end{equation}
	where the map 
	\begin{equation} \label{metric:Hermitian}
		\langle~,~\rangle: \cE \otimes_B \overline{\cE} \to B \quad  \text{ is defined by }  \langle x, \overline{y}\rangle:= \ev(x \otimes_B \kH(\overline{y})).
	\end{equation} 
 We say that $\kH$ is positive  if $ \langle e, \overline{e} \rangle \geq 0. $ Moreover,  $\kH$ is called positive definite if $\kH$ is positive  and  $ \langle e, \overline{e} \rangle = 0 $ if and only if $e = 0.$   
\end{definition}

Let $\nabla$ be a (left) connection on a $B$-bimodule $\mathcal{E}$ with $\nabla (e)= \sum_i \omega_i \otimes_B x_i, $ we define a right connection $\widetilde{\nabla}$ on $\overline{\mathcal{E}}$ as
\begin{equation} \label{10thjan254}
 \widetilde{\nabla} (\overline{e})= \sum_i \overline{x_i}\otimes_B \omega_i^{\ast}.
 \end{equation}

\begin{definition} \braces{\cite[Definition 8.33]{BeggsMajid:Leabh}} \label{19thoct231}
    A left connection $\nabla$ on a $B$-bimodule $\cE$ is said to be compatible with an Hermitian metric $\mathscr{H}$ on $\mathcal{E}$   if
	\begin{align*} 
		d\langle ~ , ~ \rangle = (\id \otimes_B \langle ~ , ~\rangle)(\nabla \otimes_B \id) + (\langle ~ , ~ \rangle \otimes_B \id) (\id\otimes_B \widetilde{\nabla})
	\end{align*}
	as maps from $ \mathcal{E} \otimes_B \overline{\mathcal{E}} $ to $ \Omega^1$.
	
	\end{definition}

\subsection{From compatibility with metrics to compatibility with Hermitian metrics}

\begin{prop} \braces{\cite[Proposition 4.3]{LeviCivitaHK}} \label{27thdec242jb}
	Let $\dc$ be a $\ast$-differential calculus on a  $\ast$-algebra $B$ such that $\Omega^1$ is finitely generated and projective as a left $B$-module. Given a real  metric $ ( g, ( ~, ~ )  ) $ on $\Omega^1,$ define
	$$ \kH_g: \overline{\Omega^1} \rightarrow \hm{B}{\Omega^1}{B} ~ {\rm as} ~ \kH_g ( \overline{\omega}  ) ( \eta ) = ( \eta, \omega^* ). $$
	Then  the correspondence $ ( g, ( ~ , ~ ) ) \mapsto \kH_g $ is a bijective correspondence between real  metrics and  Hermitian metrics on $\Omega^1.$
\end{prop}

If $(\nabla, \sigma )$ is a bimodule connection on the space of one-forms $\Omega^1$ of a $\ast$-differential calculus $\dc,$ then we have a notion of compatibility of $\nabla$ with a real metric $\metric$ (as in Definition \ref{4thjuly241}) and another notion of compatibility with the Hermitian metric $\kH_g$ (as in Definition \ref{19thoct231}). If $\sigma$ is invertible and $(\nabla, \sigma)$ is $\ast$-preserving (see Definition \ref{15thapril252}), then these two notions coincide. This is the content  of Proposition \ref{21stapril251}, whose proof is spread over several results in \cite{BeggsMajid:Leabh}. Therefore, we include a proof for the sake of completeness. We start by introducing a monoidal category from \cite[Section 3.4.2]{BeggsMajid:Leabh}.

Let $\dc$ be a differential calculus on an algebra $B.$ Consider the category   ${}_B{\mathcal{E}\mathcal{I}}_B$  with objects $ ( \cE,  \nabla, \sigma_{\cE}  ),   $  where  $( \nabla, \sigma_{\cE} )$ is a left bimodule connection on a $B$-bimodule $\cE$ with $\sigma_{\cE}$ invertible. A morphism $\theta: ( \cE,  \nabla, \sigma_{\cE}  ) \rightarrow ( \cF,  \nabla^\prime, \sigma_{\cF}  ) $ in ${}_B{\mathcal{E}\mathcal{I}}_B$ is a $B$-bilinear map from $\cE$ to $\cF$ such that $\nabla^\prime \circ \theta = ( \text{id} \otimes_B \theta ) \nabla. $
It follows that $\cateibb$ is a monoidal category with the unit object $( B,  d, \text{id} ). $

In what follows, for a $B$-bimodule $\cE,$ the symbol $\mathrm{Hom}_B ( \cE, B) $ will stand for the $B$-bimodule of  right $B$-linear maps from $\cE$ to $B.$
We recall that if $\cE$ is finitely generated and projective as a right $B$-module, then there are elements $\{ e^i: i = 1, \cdots n \}$ in $\cE$ and $\{ e_i: i = 1, \cdots n  \}$ in $\mathrm{Hom}_B ( \cE, B) $ such that 
$$ e = \sum_i e^i e_i ( e ) ~ \text{and} ~ \phi = \sum_i \phi ( e^i ) e_i. $$
In this case, consider the maps
$$ \mathrm{ev}_L: \mathrm{Hom}_B ( \cE, B)  \otimes_B \cE \rightarrow B  ~ \text{and} ~ \mathrm{coev}_L: B \rightarrow \cE \otimes_B \mathrm{Hom}_B ( \cE, B)  $$
defined as 
$ \mathrm{ev}_L ( \phi \otimes_B e ) = \phi ( e ), ~   \mathrm{coev}_L ( b ) = b \sum_i e^i \otimes_B e_i. $ 


Beggs and Majid have proved the following result:

\begin{prop} \braces{\cite[Proposition 3.8 and Proposition 3.79]{BeggsMajid:Leabh}}      \label{21stapril253}
If $( \cE,  \nabla, \sigma_{\cE}  )$ is an object in $\cateibb$ with $\cE$ finitely generated and projective as a right $B$-module, then there exists a unique left bimodule connection  on $\mathrm{Hom}_B ( \cE, B) $ such that $\text{ev}_L$ and $\text{coev}_L $ are morphisms in $\cateibb.$
\end{prop}

The following Lemma is a consequence of  Proposition \ref{21stapril253}. 

\begin{lem} \label{21stapril254}
 If $\metric$ is a metric on $\Omega^1$ and $( \Omega^1, \nabla, \sigma)   $ is an object of $\cateibb,$  then the following statements are equivalent:
 \begin{enumerate}
 \item  $(\nabla, \sigma)$ is compatible with $\metric$ in the sense of Definition \ref{4thjuly241}.
 
 \item The map $\text{coev}_g $ (defined in \eqref{23rdmay241}) is a morphism in the category $ \cateibb. $
 
 \item  $ ( ~, ~ ) : \Omega^1 \otimes_B \Omega^1 \rightarrow B $ is  a morphism in $\cateibb.$ 
 \end{enumerate}
\end{lem}
\begin{proof} Note that for $b \in B, \nabla_{\Omega^1 \otimes_B \Omega^1} ( \text{coev}_g ( b ) ) = \nabla_{\Omega^1 \otimes_B \Omega^1} ( b g ) = b \nabla_{\Omega^1 \otimes_B \Omega^1} g + db \otimes_B g. $ 
Hence,
 $$   \nabla_{\Omega^1 \otimes_B \Omega^1} ( \text{coev}_g ( b ) ) - db \otimes_{B} g = b \nabla_{\Omega^1 \otimes_B \Omega^1} g.
$$
 So, $\nabla$ is metric compatible if and only if $\coev_g$ is a morphism. This proves the equivalence of (1) and (2). 
 
For proving the equivalence of (2) and (3), we make use of Proposition \ref{21stapril253}. 
We  observe that since  $\Omega^1$ admits a metric $\metric,$   Remark \ref{24thjuly241} implies that   $( \Omega^1, ( ~, ~ ), \coev_g )$ is a left dual of $\Omega^1$ in the monoidal category $\catmodbb$ and  moreover, $\form{1}$ is  finitely generated and  projective as a  right $B$-module.
Thus, Proposition \ref{21stapril253} applied to the latter statement shows that 
 $( \mathrm{Hom}_B ( \Omega^1, B), \ev_L, \coev_L )$ is another left dual of $\form{1}$ in $\catmodbb.$ 
 
Therefore, by \cite[Proposition 2.10.5]{etingof2015tensor}, the map
\begin{equation} \label{23rdapril251}
 T:= ( \text{ev}_L \otimes_B \text{id} ) ( \text{id} \otimes_B \text{coev}_g  ) : \mathrm{Hom}_B ( \Omega^1, B) \rightarrow \Omega^1 
 \end{equation}
is the unique isomorphism in $\catmodbb$  such that
\begin{equation*} 
 \text{ev}_L = ( ~, ~ ) \circ ( T \otimes_B \text{id} ) ~ \text{and} ~ ( \text{id} \otimes_B T ) \text{coev}_L = \text{coev}_g. 
 \end{equation*}
Moreover,
\begin{equation} \label{29thjuly251}
T^{-1} = ( ( ~, ~ ) \otimes_B \text{id}  ) ( \text{id} \otimes_B \text{coev}_L ). 
\end{equation}

Now let us assume that   $\text{coev}_g$ is a morphism in $\cateibb.$ From Proposition \ref{21stapril253}, we know that $\text{ev}_L$ is a morphism in $\cateibb$ and so \eqref{23rdapril251} shows that $T$ (and hence $T^{-1}$) are morphisms in the same category. Since  $ ( ~, ~ ) = \text{ev}_L \circ ( T^{-1} \otimes_B \text{id}  ), $ we conclude that 
$( ~, ~ )$ is  a morphism in $\cateibb$. 

Conversely, if  $( ~, ~ )$ is a morphism in $\cateibb$, then from \eqref{29thjuly251}, we infer that  $T^{-1}$ (and hence $T$) are morphisms in $\cateibb.$ Thus $\coev_g = ( \text{id} \otimes_B T ) \text{coev}_L $ is a morphism in $\cateibb.$
This proves the equivalence of $(2)$ and $(3)$.
\end{proof}

\begin{prop} \braces{\cite{BeggsMajid:Leabh}} \label{21stapril251}
Let $\dc$ be a $*$-differential calculus on  a $\ast$-algebra $B.$ If  $( \form{1}, \nabla, \sigma )$ is an object of $\cateibb$ such that $(\nabla, \sigma)$ is   $\ast$-preserving, then the following statements are equivalent:
\begin{enumerate}
\item[(i)] $\nabla$ is compatible with a real metric $\metric$ in the sense of Definition \ref{4thjuly241}.

\item[(ii)] $\nabla$ is compatible with the Hermitian metric $\kH_g$ in the sense of Definition \ref{19thoct231}.
\end{enumerate}
\end{prop}
\begin{proof}
Let us assume that $\nabla_{\Omega^1 \otimes_B \Omega^1} g = 0. $ Then, by Lemma \ref{21stapril254}, $( ~, ~ )$ is a morphism in $\cateibb,$  i.e., for all $\omega, \eta \in \Omega^1,  d ( \omega, \eta ) = ( \text{id} \otimes_B ( ~, ~ ) ) \nabla_{\Omega^1 \otimes_B \Omega^1} ( \omega \otimes_B \eta ).  $
Therefore, we have
\begin{equation} \label{16thapril251}
d ( \omega, \eta ) = ( \text{id} \otimes_B ( ~, ~ ) ) ( \nabla \omega \otimes_B \eta + ( \sigma \otimes_B \text{id} ) ( \omega \otimes_B \nabla \eta )  ).
\end{equation}
Since $\nabla$ is 
$\ast$-preserving, \cite[Proposition 8.40]{BeggsMajid:Leabh} implies that $\nabla$ is compatible with $\kH_g.$ 

Conversely, let us assume that $\nabla$ is compatible with $\kH_g.$ Since $\nabla$ is $\ast$-preserving, \cite[Proposition 8.40]{BeggsMajid:Leabh} once again implies that \eqref{16thapril251} holds, i.e, $( ~, ~ )$ is a morphism in $\cateibb.$ Then by Lemma \ref{21stapril254}, it follows that $\nabla_{\Omega^1 \otimes_B \Omega^1} g  = 0. $
\end{proof}

\section{Compatibility of connections with Hermitian metrics on weak \texorpdfstring{$\dagger$}{}-bimodules} \label{21staugust251}

In the previous section,  we explained the approach to define  metric compatibility of a bimodule connection in the sense of \cite{BeggsMajid:Leabh}. On the other hand, if $B$ is a $\ast$-algebra, the authors of \cite{meslandrennie1} study compatibility of connections (not necessarily bimodule connections) on $\dagger$-bimodules (see Definition \ref{9thmay251}),  examples of which include a class of first order $\ast$-differential calculi equipped with a positive definite sesquilinear form.

In this section, we compare the above two definitions and our main result is Theorem \ref{9thmay252} that works for a class of sesquilinear forms defined on the balanced tensor product $\Omega^1 \otimes_B \Omega^1.$ However, unlike \cite{meslandrennie1}, we do not impose any positivity condition on the sesquilinear form and in fact work with a larger class of weak $\dagger$-bimodules instead of the $\dagger$-bimodules of \cite{meslandrennie1}.

Throughout this section,  $B$ is a $*$-algebra and $(\form1, d)$ is a first order $*$-differential calculus.  The authors of \cite{meslandrennie1} use a closely related notion. We make a note of their difference in the following remark.

\begin{remark}
The authors of \cite{meslandrennie1} use the term first order differential structure to refer to a notion which is almost the same as a first order $\ast$-differential calculus. A crucial difference is that they postulate the condition $d ( b^* ) = - d ( b )^{\ast}. $  Moreover, they also assume $B$ to be a local algebra in their definition. We refer to \cite[Definition 2.3]{meslandrennie1} and the discussion following it for the details. 
\end{remark}

However, the two concepts can be reconciled by the following simple observation:

\begin{prop}\label{prop:oldnewdc}
    Let $B$ be a $*$-algebra and $\cE \in \catmodbb$. Then the following are equivalent:
    \begin{enumerate}
        \item [(i)] $(\form{1},d)$ is a first order $*$-differential calculus and $\nabla$ is a connection on $\cE$ for $(\form{1},d).$
        \item [(ii)] Define $d^\prime= \ci d: B\to \form{1},$ then $(\form{1}, d^\prime)$ is a differential calculus satisfying $d^\prime(b^*) = -(d^\prime b)^*$. Moreover $\ci \nabla$ is a connection on $\cE$ for    $(\form{1},d').$
    \end{enumerate}
\end{prop}

Much of the theory developed in \cite{meslandrennie1} relies on the language of $\dagger$-structures on bimodules over $\ast$-algebras. 

\begin{definition}\label{def:daggerbimod}
 A $B$-bimodule $M$  over a $\ast$-algebra $B$ is said to have a $\dagger$-structure if there exists an antilinear involution
 \begin{equation} \label{10thmay252}
  \dagger: M \rightarrow M ~ \text{such that} ~ \dagger ( b m c ) = c^* \dagger ( m ) b^*.   
 \end{equation}
\end{definition}
If $\dagger_{\cE}$ and $\dagger_{\cF}$ are $\dagger$-structures on bimodules $\cE$ and $\cF$ respectively, then  Mesland and Rennie consider the map
\begin{equation} \label{13thmay251}
 \dagger_{\cE \otimes_B \cF}: \cE \otimes_B \cF \rightarrow \cF \otimes_B \cE  ~ \text{as} ~ \dagger_{\cE \otimes_B \cF} ( e \otimes_B f ) = \dagger_{\cF} ( f ) \otimes_B \dagger_{\cE} ( e ). 
\end{equation}
We observe that the notion of a $\dagger$-structure is intimately related with that of a star object in the bar category $\catmodbb$ (see Definition \ref{15thjuly241}).
Indeed, given a $\dagger$-structure on $M$,  we define a morphism
$ \star_{M, \dagger}: M \rightarrow \overline{M} ~ \text{by} ~ \star_{M, \dagger} ( m ) = \overline{\dagger (m) }. $
Conversely, if $( M, \star_M )$ is a star object in $\catmodbb,$ we define the antilinear map $\dagger_{\star_M}: M \rightarrow M ~ \text{by} ~$
\begin{equation} \label{14thmay256}
  \dagger_{\star_M} ( m ) = \text{bb}^{-1}_M  \overline{\star_M} ( \overline{m} ).
\end{equation}

\begin{prop}
 Let $M$ be an object in $\catmodbb.$ The correspondences $\dagger \mapsto \star_{M, \dagger}$ and $\star_M \mapsto \dagger_{\star_M} $ define a bijection between $\dagger$-structures and star object structures on the bimodule $M.$   
\end{prop}
 \begin{proof}
 If $\dagger$ is a $\dagger$-structure, then it can be easily checked by using \eqref{10thmay252} and \eqref{28thnov231} that $\star_{M, \dagger}$ is a morphism in $\catmodbb.$ The condition \eqref{10thmay253} is satisfied as $\dagger$ is an involution. 

Conversely, if $(M, \star_M)$ is a star object, then $\dagger_{\star_M}$ satisfies \eqref{10thmay252} as $\star_M$ is a morphism in $\catmodbb.$
The fact that $\dagger_{\star_M}$ is an involution follows from the following computation:
$$( \dagger_{\star_M} )^2 ( m ) = \text{bb}^{-1}_M \overline{\star_M} (  \overline{\text{bb}^{-1}_M \overline{\star_M} ( \overline{m} ) }   ) = \text{bb}^{-1}_M (  \overline{ \star_M \text{bb}^{-1}_M \overline{\star_M} ( \overline{m} ) }   ) = \text{bb}^{-1}_M ( \overline{\overline{m}} )
~ \text{(by \eqref{10thmay253})}. $$
But this equals $m$ as $\text{bb}_M ( m ) = \overline{\overline{m}} $     
by Example \ref{5thdec241jb}. 
Finally, $ \dagger_{\star_{M, \dagger}} ( m ) = \text{bb}^{-1}_M ( \overline{\star_{M, \dagger} ( m ) } ) = \text{bb}^{-1}_M ( \overline{\overline{\dagger ( m ) }} ) = \dagger ( m ) $
and if $\star_M ( m ) = \overline{n},$ then 
$ \star_{M, \dagger_{\star_M}} ( m ) = \overline{\dagger_{\star_M} ( m ) } = \overline{\text{bb}^{-1}_M \overline{\star_M} ( \overline{m} ) } = \overline{ \text{bb}^{-1}_M ( \overline{\overline{n}} ) } = \overline{n} = \star_M ( m ).  $
Therefore, the two correspondences in the statement of the proposition are inverses of one another. This completes the proof.  
 \end{proof}

Let us note that if $(\Omega^1, d)$ is a first order $\ast$-differential calculus and $\star_{\Omega^1}$ is as defined in \eqref{14thmay251}, then 
\begin{eqnarray} \label{18thjune251}
    \dagger_{\star_{\Omega^1}} ( \omega ) = \omega^\ast.
\end{eqnarray}

\subsection{The class of  weak \texorpdfstring{$\dagger$}{}-bimodules}
 A left $B$-valued sesquilinear form   on a $B$-bimodule $\mathcal{E}$ is a map
$$ {}_B \langle \langle ~ , ~ \rangle \rangle: \mathcal{E} \times \mathcal{E} \rightarrow B $$
which is left $B$-linear, right $\mathbb{C}$-antilinear, and satisfies ${}_B \langle \langle e , f \rangle \rangle = {{}_B \langle \langle f , e \rangle \rangle}^\ast $ for all $e, f \in \mathcal{E}.$
A $B$-valued sesquilinear form on $\mathcal{E}$ is called a pre-inner product if
 ${}_B \langle \langle e , e \rangle \rangle \geq 0.$ A pre inner product is called an inner product if the condition 
 ${}_B \langle \langle e , e \rangle \rangle = 0 $ holds if and only if $e = 0.$ 
In the sequel, we will use the term sesquilinear form on $\mathcal{E}$ to denote a left $B$-valued sesquilinear form on $\mathcal{E}.$
Observe that if $  {}_B \langle \langle ~ , ~ \rangle \rangle    $ is a sesquilinear form on $\mathcal{E},$  then for all $e, f \in \mathcal{E}$ and $b \in B,$ we have
\begin{equation} \label{10thjan251}
 {}_B \langle \langle e , bf \rangle \rangle = {}_B \langle \langle e , f \rangle \rangle b^*.   
\end{equation}

\begin{definition} \braces{\cite[Definition 2.8]{meslandrennie1}} \label{9thmay251}
A weak $\dagger$-bimodule over $B$ is a triplet $(X, \dagger, {}_B\langle \langle ~ , ~ \rangle \rangle)$ where $X$ is a  $B$-bimodule which is finitely generated and projective as a left $B$-module, $\dagger$ is a $\dagger$-structure on $X$ as in Definition \ref{def:daggerbimod} and  $ {}_B\langle \langle ~ , ~ \rangle \rangle: X \times X \rightarrow B$ is a sesquilinear form such that
\begin{equation} \label{3rdjune252}
 {}_B\langle \langle x b , y \rangle \rangle = {}_B\langle \langle x , y b^* \rangle \rangle  
\end{equation}
for all $x, y \in X, b \in B.$ If ${}_B\langle \langle ~ , ~ \rangle \rangle$ is an inner product, then we say that $X$ is a $\dagger$-bimodule.
\end{definition}

 The class of weak $\dagger$-bimodules which will be of interest to us arise from Hermitian metrics (as in Definition \ref{8thmay252}) on  first order $\ast$-differential calculi. We make the following observation:

\begin{prop} \label{10thjan2511}
Let  $\kH$ be an Hermitian metric on a $B$-bimodule $\mathcal{E}.$  Then, 
\begin{equation} \label{10thjan259} 
 {}_B \langle \langle ~ , ~ \rangle \rangle: \mathcal{E} \times \mathcal{E} \rightarrow B ~ {\rm defined} ~ {\rm by} ~ {}_B \langle \langle e , f \rangle \rangle = \kH ( \overline{f} ) ( e ) 
 \end{equation}
is a sesquilinear form on $\mathcal{E}$ satisfying \eqref{3rdjune252}. 

If $B$ is a $\ast$-subalgebra of a $C^*$-algebra  and $\kH$ is a positive (resp, positive definite) Hermitian metric, then ${}_B \langle \langle ~ , ~ \rangle \rangle$ is a pre-inner product (resp, inner product) on $\mathcal{E}.$  
\end{prop}
\begin{proof}
If $\langle ~, ~ \rangle $ denotes the map defined in \eqref{metric:Hermitian}, then
\begin{equation} \label{10thjan256}
{}_B \langle \langle e , f \rangle \rangle = \langle e, \overline{f} \rangle.
\end{equation} Thus $\langle\langle~,~\rangle\rangle$ is left $B$-linear.
 Next, by using \eqref{15thdec246}, we get  
$
 {}_B \langle \langle e , f \rangle \rangle = \langle e, \overline{f} \rangle = (  \langle f, \overline{e} \rangle  )^\ast = {}_B \langle \langle f , e \rangle \rangle^\ast,
$
 and hence ${}_B \langle \langle ~ , ~ \rangle \rangle$ is a sesquilinear form. The equation \eqref{3rdjune252} is easily verified by using the fact that $\langle ~, ~ \rangle $ is defined over $\cE \otimes_B \overline{\cE}.$ 
 
Now we assume that $B$ is a $\ast$-subalgebra of a $C^*$-algebra. If $\kH$ is positive and $e \in \mathcal{E},$ then 
${}_B \langle \langle e , e \rangle \rangle = \langle e, \overline{e} \rangle \geq 0. $
Thus, ${}_B \langle \langle ~ , ~ \rangle \rangle$ is a left $B$-valued pre-inner product module.
Finally, assume that $\kH$ is positive definite and suppose that $e \in \mathcal{E}$ is such that ${}_B \langle \langle e , f \rangle \rangle = 0 $ for all $f \in \mathcal{E}.$ Then $ \kH ( \overline{e} ) ( f ) = {}_B \langle \langle f , e \rangle \rangle = 0 $
for all $ f \in \mathcal{E} $. Since $\kH$ is injective, we see that $e = 0.$
Therefore, if $e \in \mathcal{E}$ is such that ${}_B \langle \langle e , f \rangle \rangle = 0 $ for all $f \in \mathcal{E},$ then $e = 0.$ Then  by using \cite[Proposition 1.1]{lance}, it easily follows  that ${}_B \langle \langle e , e \rangle \rangle = 0 $ if and only if $e = 0.$
\end{proof}

\begin{cor} \label{10thjan2510}
Suppose that $(\Omega^1, d)$ is a first order $\ast$-differential calculus and let us equip $\form{1}$ with the $\dagger$-structure $\dagger_{\star_{\form{1}}}$ as in  \eqref{14thmay256}.    If $\kH$ is an Hermitian metric on $\Omega^1$,  then $ ( \Omega^1,\dagger_{\star_{\form{1}}} , {}_B \langle \langle ~ , ~ \rangle \rangle )  $ is a weak $\dagger$-bimodule where  ${}_B \langle \langle ~ , ~ \rangle \rangle $ is defined by \eqref{10thjan259}.
\end{cor}
\begin{proof}
By virtue of Proposition \ref{27thdec242jb}, there exists a real metric $\metric$ on $\Omega^1$ such that $\kH = \kH_g.$ Then by Remark \ref{24thjuly241}, $\Omega^1$ is finitely generated and projective as a left $B$-module. Therefore, the conclusion follows from Proposition \ref{10thjan2511}.   
\end{proof}

Our next result  is a  partial converse to Proposition \ref{10thjan2511}. Let us recall that if 
 ${}_B\langle \langle ~ , ~ \rangle \rangle$ is a left inner product on a $B$-bimodule $X,$  then a finite  frame for $X$  is a finite collection of elements $\{ x_j \}_{j=1}^k$ such that for all $x \in X,$
$ x = \sum_{j=1}^k {}_B\langle \langle x , ~ x_j \rangle \rangle x_j. $

\begin{prop} \label{10thmay251}
Let $ ( \mathcal{E}, {}_B \langle \langle ~ , ~ \rangle \rangle ) $ be a $\dagger$-bimodule that admits a finite frame $ \{ x_j \}_j. $ Then, 
$ \kH: \overline{\mathcal{E}} \rightarrow \hm{B}{\cE}{B} ~ {\rm defined} ~ {\rm by} ~ \kH ( \overline{e} ) ( f ) = {}_B \langle \langle f , e \rangle \rangle $
is a positive definite Hermitian metric on $\cE.$ 
\end{prop} 
\begin{proof}
To begin with, the left $B$-linearity of ${}_B \langle \langle ~ , ~ \rangle \rangle$ implies that  $ \kH ( \overline{e} ) $ is a left $B$-linear map for all $e \in \mathcal{E}.$ 
Now, by applying \eqref{10thjan251} and  \eqref{19thdec234}, we obtain
$ \kH ( \overline{e} b ) ( f ) = \kH ( \overline{b^* e} ) ( f ) = {}_B \langle \langle f , b^* e \rangle \rangle = {}_B \langle \langle f , e \rangle \rangle b = \kH ( \overline{e} ) ( f ) b = ( \kH ( \overline{e} ) b ) ( f ), $
which proves that the map $\kH$ is right $B$-linear. Similarly, the left $B$-linearity of $\kH$ follows from \eqref{3rdjune252} and \eqref{19thdec234}.

 For proving that $\kH$ is surjective, we assume $ f \in \hm{B}{\cE}{B}  $ and define  $ x:= \sum_j f ( x_j )^\ast x_j. $
Then
\begin{eqnarray*}
\kH ( \overline{x} ) ( e ) &=& {}_B \langle \langle e , x \rangle \rangle = \sum_j {}_B \langle \langle e , f ( x_j )^\ast x_j \rangle \rangle\\
&=& \sum_j {}_B \langle \langle e , x_j \rangle \rangle f ( x_j ) = f ( \sum_j {}_B \langle \langle e , x_j \rangle \rangle x_j ) = f ( e )
\end{eqnarray*}
and we have applied \eqref{10thjan251}. This proves that $\kH$ is surjective. 
 Moreover,
$ \langle e, \overline{e} \rangle = {}_B \langle \langle e , e \rangle \rangle \geq 0. $
If $ \langle e, \overline{e} \rangle = 0, $ then $ {}_B \langle \langle e , e \rangle \rangle = 0 $ proving that $e = 0.$ Therefore, $\langle ~, ~ \rangle$ is positive definite and $\kH$ is injective. Finally, for all $e, f \in \mathcal{E},$
$ \langle e, \overline{f} \rangle = {}_B \langle \langle e , f \rangle \rangle = ( {}_B \langle \langle f , e \rangle \rangle  )^\ast  = \langle f, \overline{e} \rangle^\ast. $
\end{proof}

By \cite[Corollary 2.11]{meslandrennie1}, the hypothesis of  Proposition \ref{10thmay251} is satisfied if  $B$ is a dense unital local $\ast$-subalgebra of a unital $C^*$-algebra and $\cE$ is finitely generated and projective as a left $B$-module.

\subsection{Hermitian connections on weak \texorpdfstring{$\dagger$}{}-bimodules}

Let $B$ be a $\ast$-algebra and $(\Omega^1, d^\prime)$ a first order differential calculus on $B$ such that $\Omega^1$ is equipped with a $\dagger$-structure satisfying the equation $ d^\prime ( b^\ast ) = - \dagger ( d^\prime ( b ) ). $ Furthermore, we assume that $\cE$ is a $B$-bimodule equipped with a sesquilinear form $ {}_B \langle \langle ~, ~ \rangle \rangle. $ Let us recall \cite[equation (2.6)]{meslandrennie1} the  pairings
$ \mathcal{E} \times ( \Omega^1 \otimes_B \mathcal{E}  ) \rightarrow \Omega^1  $ and $ ( \Omega^1 \otimes_B \mathcal{E}  ) \times \mathcal{E} \rightarrow \Omega^1 $ defined by
\begin{equation} \label{10thjan257}
 {}_{\Omega^1} \langle \langle e , \omega \otimes_B f \rangle \rangle := {}_B \langle \langle e , f \rangle \rangle \dagger ( \omega ),\quad {}_{\Omega^1} \langle \langle \omega \otimes_B  e , f \rangle \rangle := \omega {}_B \langle \langle e , f \rangle \rangle  
\end{equation}
for all $e, f \in \mathcal{E}$ and for all $\omega \in \Omega^1.$ The following definition coincides with \cite[Definition 2.23]{meslandrennie1} when $\mathcal{E}$ is a $\dagger$-bimodule.

\begin{definition} (\cite{meslandrennie1}) \label{21staugust252}
Suppose that $(\Omega^1, d^\prime)$ is a first order differential calculus on a $\ast$-algebra $B$ and  $\dagger_{\form{1}}$ be  a $\dagger$-structure on $\Omega^1$ such that $d^\prime ( b^* ) = - d^\prime ( b )^\dagger $ for all $b \in B.$    A left connection $\nabla^\prime$ on a $B$-bimodule $\cE$ is said to be Hermitian   with respect to a sesquilinear form $ {}_B \langle \langle ~ , ~ \rangle \rangle  $ on $\cE$ if
\begin{equation} \label{10thjan255}
d^\prime ( {}_B \langle \langle e , f \rangle \rangle  ) = {}_{\Omega^1} \langle \langle \nabla^\prime ( e ) , f \rangle \rangle - {}_{\Omega^1} \langle \langle e , \nabla^\prime ( f ) \rangle \rangle.
\end{equation}
\end{definition} 

In this subsection, we show that  connections which are compatible with a real metric in the sense of \cite{BeggsMajid:Leabh} are in one to one correspondence with Hermitian connections with respect to a canonical choice of a weak $\dagger$-bimodule structure on $\Omega^1,$ provided the connections are $\ast$-preserving to begin with. 

In the following lemma, we tacitly use the correspondence of Proposition \ref{prop:oldnewdc}.

\begin{lem} \label{10thjan2512}
Let  $\kH$ be an Hermitian metric on a $B$-bimodule $\cE.$ A left connection $\nabla$ for a $\ast$-differential calculus $(\Omega^1, d)$ on $\mathcal{E}$ is compatible with $\kH$ in the sense of Definition \ref{19thoct231} if and only if the connection $ \sqrt{-1}\nabla$ for the differential calculus $(\Omega^1, \sqrt{-1} d)$ is Hermitian with respect to the sesquilinear form $ {}_B \langle \langle ~ , ~ \rangle \rangle  $ defined in Proposition \ref{10thjan2511}.  
\end{lem} 
\begin{proof}
Note that  $\Omega^1$ is a star object in $\catmodbb$ and hence it is equipped with the $\dagger$-structure $\dagger_{\star_{\Omega^1}}$ as in \eqref{18thjune251}. Thus, we are allowed to  use the defining equations \eqref{10thjan257} in the proof. For $e, f \in \cE$ assume
$ \nabla ( e ) = \sum_i \omega_i \otimes_B e_i, ~ \nabla ( f ) = \sum_j \eta_j \otimes_B f_j. $
Then, we have the following:
\begin{eqnarray*}
&& \sqrt{-1} d ( {}_B \langle \langle e , f \rangle \rangle ) - {}_{\Omega^1} \langle \langle \sqrt{-1} \nabla ( e ), f \rangle \rangle + {}_{\Omega^1} \langle \langle e , \sqrt{-1} \nabla ( f ) \rangle \rangle\\
&=& \sqrt{-1} \Big( d ( {}_B \langle \langle e , f \rangle \rangle ) - {}_{\Omega^1} \langle \langle \nabla e , f \rangle \rangle - {}_{\Omega^1} \langle \langle e , \nabla f \rangle \rangle \Big) \\
&=& \sqrt{-1} \Big(  d ( \langle e, \overline{f} \rangle  ) - \sum_i {}_{\Omega^1} \langle \langle \omega_i \otimes_B e_i, f \rangle \rangle - \sum_j {}_{\Omega^1} \langle \langle e , \eta_j \otimes_B f_j \rangle \rangle \Big) ~ {\rm (}  {\rm by} ~ \eqref{10thjan256}  {\rm)}\\
&=& \sqrt{-1} \Big( d ( \langle e, \overline{f} \rangle  ) - \sum_i \omega_i {}_B \langle \langle e_i , f \rangle \rangle - \sum_j {}_B \langle \langle e , f_j \rangle \rangle \dagger_{\star_{\Omega^1}} ( \eta_j ) \Big) \\
&=& \sqrt{-1} \Big( d ( \langle e, \overline{f} \rangle  ) - \sum_i \omega_i \langle e_i, \overline{f} \rangle - \sum_j \langle e, \overline{f_j} \rangle \eta^*_j \Big) ~ {\rm (}  {\rm by} ~ \eqref{18thjune251} ~ {\rm and} ~ \eqref{10thjan256} ~  {\rm)}  \\
&=& \sqrt{-1} \Big( d ( \langle e, \overline{f} \rangle  ) - ( {\rm id} \otimes_B \langle ~, ~ \rangle  ) ( \nabla ( e ) \otimes_B \overline{f} ) - ( \langle ~, ~ \rangle \otimes_B {\rm id}  ) (  {\rm id} \otimes_B \widetilde{\nabla} ) ( e \otimes_B \overline{f} ) \Big)
\end{eqnarray*}
by \eqref{10thjan254}. This concludes the proof.
\end{proof}

Finally, we present the main result of this section, where for a real metric $\metric$ (see Definition \ref{23rddec24jb2}) on $\form{1},$  $\kH_g$ will denote the Hermitian metric on $\Omega^1$ defined in Proposition \ref{27thdec242jb}.

\begin{theorem} \label{9thmay252}
Let $\metric$ be a metric on the space of one forms $\form{1}$ of a $\ast$-differential calculus $\dc$ and  $( \nabla, \sigma )$ be a $\ast$-preserving bimodule connection on $\Omega^1$ such that $\sigma$ is an isomorphism. Then, the following statements are equivalent:
    \begin{enumerate}
        \item[(i)] The connection $\nabla$ is compatible with $\metric$ in the sense of Definition \ref{4thjuly241}.

    \item[(ii)] For the weak $\dagger$-bimodule $(\Omega^1,  \dagger_{\star_{\Omega^1}}, {}_B \langle \langle ~, ~ \rangle \rangle  )$ obtained from the Hermitian metric $\kH_g$ as in Corollary \ref{10thjan2510}, the connection $ \sqrt{-1} \nabla$ for the differential calculus $(\Omega^1, \sqrt{-1} d)$ is Hermitian.   
    \end{enumerate}
\end{theorem}
\begin{proof}
The proof follows by combining Proposition \ref{21stapril251} and Lemma \ref{10thjan2512}. Indeed, if $\nabla g = 0,$ then  Proposition \ref{21stapril251} implies that $\nabla$ is compatible with $\kH_g$ in the sense of Definition \ref{19thoct231}.  But then, Lemma \ref{10thjan2512} implies that \eqref{10thjan255} holds for $d^\prime = \sqrt{-1} d $ and $\nabla^\prime = \sqrt{-1} \nabla. $ The converse direction follows similarly.   
\end{proof}

\section{Torsion of a connection for weak second order differential structures} \label{19thjune251}

This section compares the definitions of the torsion zero condition of connections given in \cite{BeggsMajid:Leabh} and \cite{meslandrennie1}. We have seen in Definition \ref{9thjan252} that the definition of torsion in the framework of Beggs and Majid involves  a second order differential calculus. On the other hand, Mesland and Rennie work with second order differential structures which are  first order $\ast$-differential calculi equipped with an idempotent $\psi: \Omega^1 \otimes_B \Omega^1 \rightarrow \Omega^1 \otimes_B \Omega^1 $ satisfying some conditions.
However, for the purpose of comparison, we do not need any $\ast$-structure on the calculus and hence we will deal with weak second order differential structures (see Definition \ref{5thmay252}).
At this point, we should also mention that for connections on the space of one-forms on a spectral triple (\cite{connes}), a spectral definition of torsion has been studied, for which we refer to \cite{dabrowskisitarzadv}.

   In Subsection \ref{14thjune252}, we show that if we start with a class of weak second order differential structures, then torsion zero connections in the sense of \cite{meslandrennie1} are  torsion zero connections (in the sense of \cite{BeggsMajid:Leabh}) for the associated Connes' calculus.  In Subsection \ref{14thjune253}, we prove similar results for quasi-tame differential calculi. Finally,  combining  results from previous sections, we prove that for quasi-tame $\ast$-differential calculi, Levi-Civita connections for real metrics in the sense of \cite{BeggsMajid:Leabh} are in one to one correspondence with Levi-Civita connections in the sense of \cite{meslandrennie1}. 

Let us begin by introducing some notations from \cite{meslandrennie1}. Let $( \Omega^1, d )$ be a first order differential calculus. Then we will use the notations 
$$ T^k_d ( B ) := \underbrace{\Omega^1( B ) \otimes_B \cdots \otimes_B \Omega^1 ( B )}_{k  \text{-times}}, ~ T^\ast_d ( B ) := \oplus_{k \geq 0} T^k_d ( B ).  $$
If $ \Omega^k_u, \Omega^\ast_u $ are the universal differential forms defined in Example \ref{6thmay252} and 
$\pi^{\otimes k}: \Omega^k_u ( B ) \rightarrow T^k_d ( B )$ is the surjective map given by 
$$ \pi^{\otimes k} ( a_0 \delta ( a_1 ) \cdots \delta ( a_k )  ) = a_0 da_1 \otimes_B \cdots \otimes_B da_k, $$
then we have the map
$$ \widehat{\pi}:= \oplus_k \pi^{\otimes k}: \Omega^\ast_u ( B ) \rightarrow T^\ast_d ( B ). $$ 
The sub-bimodule $JT^k_d \subseteq T^k_d $ defined as 
$$ JT^k_d = \{ \widehat{\pi} ( \delta ( \omega ) ): \omega \in \Omega^{k - 1}_u \cap \text{Ker} ( \widehat{\pi} )  \} $$
is known as the junk tensors of degree $k$.

Now, let us assume that $(\Omega^1, d)$ is a first order differential calculus such that $\Omega^1$ is a sub-bimodule of an algebra $M.$ Consider the maps 
$$ m_k: T^k_d \rightarrow M ~ \text{defined by } ~ m_k ( \omega_1 \otimes_B  \cdots \otimes_B \omega_k  ) = \omega_1 \cdots \omega_k.  $$
An  analogous  construction of Connes' differential calculus for spectral triples yields a differential calculus for $M.$  

\begin{definition} (\cite[Subsection 3.2]{meslandrennie1})  \label{20thmay251}
Suppose that $(\Omega^1, d)$ is a first order differential calculus such that $\Omega^1$ is a sub $B$-bimodule of an algebra $M.$ Then, the Connes type junk forms for the triple $(\Omega^1, d, M)$ are defined as 
$$ J^k_d = \{ m_k \circ \widehat{\pi} ( \delta ( \omega ) ): \omega \in \Omega^{k - 1}_u \cap \text{Ker} ( m_{k - 1} \circ \widehat{\pi})   \} \subseteq \text{Ran} ( m_k ). $$
Define $\Omega^k_C:= \text{Ran} ( m_k )/J^k_d, \Omega^\bullet_C:= \oplus_{k} \Omega^k_C, $  
$$  \wedge_C ( (  \omega + J^k_d  ) \otimes_B ( \eta + J^l_d  )  ) = \omega . \eta + J^{k + l}_d,\,\, d_C ( b_0 db_1 \cdots db_k + J^k_d ) =  db_0 db_1 \cdots db_k + J^{k + 1}_d $$
for all $\,\omega \in \text{Ran} ( m_k ), \eta \in \text{Ran} ( m_l ), $ then  $(\Omega^\bullet_C, \wedge_C, d_C)$ is called the Connes' differential calculus for the data $(\Omega^1, d, M).$ 
\end{definition}

In  (\cite{connes}), Connes defined differential calculus to any spectral triple $(B, H, D).$ In the above definition, if $M = B ( H ) $ we recover  the Connes' calculus. 

Let us note the following simple observation.

\begin{prop} \label{10thjune251}
Let $\dc$ be a differential calculus on an algebra $B.$ Then, the Connes' calculus  for $(\Omega^1, d)$ with $M=\Omega^\bullet$ is $\dc$ itself.
\end{prop}
\begin{proof}
Note that the multiplication in $\Omega^\bullet$ is given by the $\wedge$ map.
Let $q_k = m_k \circ \widehat{\pi} $ denote the canonical  quotient map $\uform{k}\to \form{k}.$  For all $\omega \in \Omega^k_u ( B ),$ we have
\begin{equation} \label{21stmay251}
d \circ q_k(\omega) = m_{k + 1} \circ \widehat{ \pi}\circ \delta(\omega).    
\end{equation}
 Let $x\in J^k_d.$ Then $x= m_k (\widehat{\pi}(\delta(\omega)))$ for some $\omega\in \uform{k-1}$ such that $m_{k-1}(\widehat{\pi}(\omega))=0.$ Hence, by using \eqref{21stmay251}, we get 
$$x= m_k(\widehat{\pi}(\delta(\omega)))  = d(q_{k-1}(\omega))= d(m_{k-1}(\widehat{\pi}(\omega)))=d(0)=0.$$ Then $\Omega^k_C= \mathrm{Ran}(m_k)= \form{k}$ and $\wedge_C= \wedge.$ Finally,
\begin{align*}
    d_C(b_0 db_1 \wedge \cdots \wedge db_{k-1})= m_k(db_0 \ot_B \cdots \ot_B db_{k-1})= db_0 \wedge \cdots db_{k-1}=d(  b_0 db_1 \wedge \cdots \wedge db_{k-1})
\end{align*}
which shows that $d_C=d.$ This completes the proof.
\end{proof}

\subsection{Weak second order differential structures} \label{14thjune252}
The notion of torsion in the framework of Mesland and Rennie requires the following definition which is a weakening of \cite[Definition 3.4]{meslandrennie1}.

\begin{definition} \label{5thmay252}
A weak second order differential structure on $B$ is a triplet $ ( \Omega^1, d, \psi )  $ where $( \Omega^1, d )$ is a  first order differential calculus on $B$
  and $\psi: \Omega^1 \otimes_B \Omega^1 \rightarrow \Omega^1 \otimes_B \Omega^1 $ is a $B$-bilinear idempotent such that  $JT^2_d \subseteq {\rm Ran} ( \psi ). $
\end{definition}

If $ ( \Omega^1, d, \psi ) $ is a weak second order differential structure, consider the map (see \cite[equation (3.1)]{meslandrennie1}):
\begin{equation} \label{9thjune252}
 d_\psi: \Omega^1 \rightarrow \Omega^1 \otimes_B \Omega^1, ~ d_\psi ( \omega ) = ( 1 - \psi ) ( \widehat{\pi} \circ \delta \circ \pi^{-1}  ) ( \omega )
 \end{equation}
which is well-defined because of the assumption $JT^2_d \subseteq {\rm Ran} ( \psi )$ and the idempotence of $\psi.$ 

\begin{definition} \braces{\cite[Definition 4.3]{meslandrennie1}}\label{9thjune251}
A left connection $\nabla$ on $\form{1}$ for a weak second order differential structure $ ( \Omega^1, d, \psi ) $  is called torsionless in the sense of \cite{meslandrennie1} if 
$ ( 1 - \psi ) \nabla = d_{\psi}. $
\end{definition}

For the rest of this subsection, we will concentrate on the class of weak second order differential structures where $\Omega^1$ is a sub-bimodule of a $B$-bimodule $M$ and $M$ is an algebra. Then it is easily checked that  
\begin{equation} \label{7thjan254}
 m_2 ( JT^2_d ) = J^2_d.
 \end{equation} 
Moreover, from \cite[Proposition 3.14]{meslandrennie1}, we have  
\begin{equation} \label{6thjan252}
 \wedge_C \circ d_\psi  = d_C: \Omega^1 \rightarrow \Omega^2_C.
 \end{equation}

\begin{lem} \label{6thjan251}
Suppose that $ ( \Omega^1, d,  \psi  ) $ is a weak second order differential structure such that $\Omega^1$ is a sub-bimodule of a $B$-bimodule $M$ and $M$ is an algebra. Then the following statements are equivalent:
\begin{enumerate}
\item[$(i)$] $ {\rm Ran} ( \psi ) \subseteq m^{-1}_2 ( J^2_d ).  $
\item[$(ii)$]  $ {\rm Ran} ( m_2 \circ \psi ) = J^2_d.  $ 
\item[$(iii)$]  The map $ \wedge_C \circ \psi: \Omega^1 \otimes_B \Omega^1 \rightarrow \Omega^2_C $ is identically zero, i.e, $ {\rm Ran} ( \psi ) \subseteq {\rm Ker} ( \wedge_C ). $
\end{enumerate}
\end{lem}
\begin{proof}
The equivalence of $(i)$ and $(ii)$ is already observed in \cite[page 21]{meslandrennie1}. If $(ii)$ holds and $q_{J^2_d}: \text{Ran} ( m_2 ) \rightarrow \Omega^2_C $ denotes the quotient map, then for all $x \in \Omega^1 \otimes_B \Omega^1,$ we obtain 
$$ \wedge_C \circ \psi ( x ) = q_{J^2_d} \circ m_2 \circ \psi ( x ) \in  q_{J^2_d} ( J^2_d ) = 0.  $$
Finally, let us assume that $(iii)$ holds. Then $ q_{J^2_d} \circ m_2 \circ \psi ( x ) = \wedge_C \circ \psi ( x ) = 0 $  for all $x \in \Omega^1 \otimes_B \Omega^1$  
so that $m_2 \circ \psi ( x ) \in {\rm Ker} ( q_{J^2_d} ) = J^2_d. $
This proves that $ {\rm Ran} ( m_2 \circ \psi ) \subseteq J^2_d. $
On the other hand, $ J^2_d = m_2 ( JT^2_d ) \subseteq {\rm Ran} ( m_2 \circ \psi ) $ by virtue of \eqref{7thjan254}.
\end{proof}

Now we are in a position to compare Definition \ref{9thjune251} for a weak second order differential structure and Definition \ref{9thjan252} for the associated Connes' calculus. 

\begin{prop} \label{7thjan251}
Suppose  that $ ( \Omega^1, d, \psi  ) $ is a weak second order differential structure such that $\Omega^1$ is a sub-bimodule of a $B$-bimodule $M$ and $M$ is an algebra. Moreover, assume that $\nabla$ is a connection of $\Omega^1.$ Then the following statements hold:
\begin{enumerate}
\item[$(i)$] If $ {\rm Ker} ( \wedge_C ) \subseteq {\rm Ran} ( \psi ) $ and  $\wedge_C \circ \nabla = d_C,$ then  $ ( 1 - \psi ) \nabla = d_\psi. $

\item[$(ii)$] On the other hand,  if the equivalent  conditions of Lemma \ref{6thjan251} are satisfied and  $ ( 1 - \psi ) \nabla = d_\psi, $ then $\wedge_C \circ \nabla = d_C.$
\end{enumerate}
\end{prop}
\begin{proof}
For the first assertion, we observe that
$$ \wedge_C \circ ( \nabla - d_\psi ) = \wedge_C \circ \nabla - \wedge_C \circ d_\psi = \wedge_C \circ \nabla - d_C = 0  $$
by using \eqref{6thjan252}. Therefore $ {\rm Ran} ( \nabla - d_\psi ) \subseteq {\rm Ker} ( \wedge_C ) \subseteq {\rm Ran} ( \psi ). $
Hence, if $x \in \Omega^1,$ then
$ ( \nabla - d_\psi ) ( x ) = \psi ( y ) $ for some $ y \in \Omega^1 \otimes_B \Omega^1.$ This implies that
\begin{equation} \label{7thjan252}
( 1 - \psi ) ( \nabla ( x ) - d_\psi ( x )  ) = ( 1 - \psi ) \psi ( y ) = 0.
\end{equation}
Since $ {\rm Ran} ( d_\psi ) \subseteq {\rm Ran} ( 1 - \psi ) $ (see \eqref{9thjune252}) and $1 - \psi$ is an idempotent, we deduce that
$ d_\psi ( x ) = ( 1 - \psi ) d_\psi ( x ) = ( 1 - \psi ) \nabla ( x ) $
by applying \eqref{7thjan252}.

For proving the second assertion, we compute
\begin{eqnarray*}
&& \wedge_C \circ \nabla = \wedge_C \circ ( 1 - \psi ) \circ \nabla + \wedge_C \circ \psi \circ \nabla
                     = \wedge_C \circ ( 1 - \psi ) \circ \nabla + 0 ~ {\rm (}  {\rm by} ~ {\rm Lemma} ~ \ref{6thjan251}  {\rm)}\\
                     &=& \wedge_C \circ d_\psi  = d_C 
                     \end{eqnarray*}
                    by  \eqref{6thjan252}.
\end{proof}

\subsection{Quasi-tame differential calculi as weak  second order differential structures} \label{14thjune253}

We recall that if $\dc$ is a differential calculus and $i: {\rm Ker} ( \wedge ) \rightarrow \Omega^1 \otimes_B \Omega^1 $ denotes the inclusion map, then the following is a short-exact sequence of $B$-bimodules:
\begin{equation} \label{4thjan251}
 0 \rightarrow {\rm Ker} ( \wedge ) \xrightarrow{i} \Omega^1 \otimes_B \Omega^1 \xrightarrow{\wedge} \Omega^2 \rightarrow 0.
 \end{equation}
 Following \cite[Definition 2.3]{article2} (also see Definition \cite[2.11]{article1}), we make the following definition:
 \begin{definition}
 We say that a differential calculus $\dc$ is quasi-tame if $\Omega^1$ is finitely generated and projective as a left $B$-module and moreover,  there exists a splitting of the short exact sequence in \eqref{4thjan251}, i.e, there exists a  $B$-bilinear map $s: \Omega^2 \rightarrow \Omega^1 \otimes_B \Omega^1$ such that $\wedge \circ s = {\rm id}_{\Omega^2}.$
\end{definition}

The splitting of the map $\wedge $ was considered by Beggs and Majid for defining the Ricci curvature of a connection (\cite[page 574]{BeggsMajid:Leabh}). In \cite[Theorem 2.13]{article1}, it was proved that  a quasi-tame calculus always admits a torsionless connection. 

Let $ ( \Omega^\bullet, \wedge, d, s ) $ be a quasi-tame differential calculus. Consider the  idempotent 
\begin{equation} \label{eq:psi_S}
    \psi_s:= 1 - s \circ \wedge: \Omega^1 \otimes_B \Omega^1 \rightarrow \Omega^1 \otimes_B \Omega^1.
\end{equation}

Then we have the following result:

\begin{prop}  \label{9thjan251} 
Let $ ( \Omega^\bullet, \wedge, d) $ be a  differential calculus such that $\Omega^1$ is finitely generated and projective as a left $B$-module. Then $s \mapsto \psi_s $ defines a bijective correspondence between quasi-tame differential structures on $\dc$ and weak second order differential structures $(\form{1}, d, \psi) $ satisfying $\text{Ker} (\wedge ) = \text{Ran} ( \psi ). $
\end{prop}
\begin{proof}
Although the statement essentially follows from \cite[Lemma 3.5]{meslandrennie1}, we give a proof for the sake of completeness. 

Let $(\Omega^\bullet, \wedge, d, s)$ be a quasi-tame differential calculus. Since $s \circ \wedge$ is an idempotent with $\mathrm{Ker} (s \circ \wedge)= \mathrm{Ker}(\wedge),$  
$$\mathrm{Ran}(\psi_s)=\mathrm{Ker}(s \circ \wedge)= \mathrm{Ker} ( \wedge ). $$
The proof of the inclusion $JT^2_d \subseteq \text{Ran} ( \psi_s ) $ follows along the lines of the proof of Proposition \ref{10thjune251}. Indeed, if $q_2 = \wedge \circ \widehat{\pi}: \Omega^2_u \rightarrow \Omega^2 $ denotes the quotient map as in that proposition, then for $\omega \in \Omega^1_u \cap \text{Ker} ( \widehat{\pi} ), $ we get
$$ \wedge ( \widehat{\pi} ( \delta ( \omega )  ) ) = d ( q_2 ( \omega ) ) = d ( \wedge \circ \widehat{\pi} ( \omega ) ) = 0  $$
implying that 
$$ JT^2_d \subseteq \text{Ker} ( \wedge ) = \text{Ker} ( s \circ \wedge ) = \text{Ran} ( 1 - s \circ \wedge ) = \text{Ran} ( \psi_s ).  $$
 Thus $(\form{1}, d, \psi)$ is a weak second order differential calculus.

Conversely, let $(\Omega^1, d, \psi)$ be a weak second order differential structure such that $\text{Ker} ( \wedge ) = \text{Ran} ( \psi ). $ Then it is easy to observe that the map 
$\wedge: \text{Ran} ( 1 - \psi ) \rightarrow \Omega^2 $ is an isomorphism. 
Then we define $s_\psi: \Omega^2 \rightarrow \Omega^1 \otimes_B \Omega^1 $ to be the inverse of the map $\wedge|_{\text{Ran} ( 1 - \psi )}.$ Then $(\Omega^\bullet, \wedge, d, s)$ is a quasi-tame differential calculus. 

Thus, we are left to prove that the correspondences $s \mapsto \psi_s$ and $\psi \mapsto s_\psi$ are inverses of each other. Suppose that $(\Omega^1, d, \psi)$ is a weak second order differential structure satisfying $\text{Ker} (\wedge) = \text{Ran} (\psi). $  If $x \in \text{Ran} ( 1 - \psi ) $ and $y \in \text{Ran} ( \psi ), $ we get
\begin{eqnarray*}
\psi_{s_\psi} ( x + y ) &=& ( 1 - s_\psi \circ \wedge ) ( x + y )\\
 &=& (  x - s_\psi \circ \wedge ( x )  ) + ( y - s_\psi \circ \wedge ( y ) )\\
 &=& ( x - x ) + (  y - s_\psi ( 0 )  )  ~ \text{(As } ~ s_\psi= (\wedge|_{\mathrm{Ran} (1-\psi)})^{-1}~ \text{and }  \text{Ran} ( \psi ) = \text{Ker} ( \wedge ) \text{)}\\
 &=& \psi ( x ) + \psi ( y )\\
 &=& \psi ( x + y ).
\end{eqnarray*}
Therefore, $\psi_{s_\psi} = \psi. $ Similarly, if $(\Omega^\bullet, \wedge, d, s)$ is a quasi-tame differential calculus, then it can be easily checked that $s_{\psi_s} = s. $
\end{proof}

This brings us to the main result of this subsection. 

\begin{theorem} \label{14thjan251}
Let $ ( \Omega^\bullet, \wedge, d,  s ) $ be a quasi-tame differential calculus and  $\nabla$ is a connection on $\Omega^1,$ then the following statements are equivalent:
\begin{enumerate}
\item[$(i)$] $\nabla$ is torsionless in the sense of Definition \ref{9thjan252}.

\item[$(ii)$] $\nabla$ is torsionless in the sense of Definition \ref{9thjune251} for the weak second order differential structure  $ ( \Omega^1, d, \psi_s ) $  obtained from Proposition \ref{9thjan251}.
\end{enumerate}
\end{theorem}
\begin{proof}
Let us assume that $\wedge \circ \nabla = d.$ We denote the Connes' calculus for the inclusion $\Omega^1 \subseteq \Omega^\bullet$ by  $(\Omega^\bullet_C, \wedge_C, d_C).$ Then by Proposition \ref{10thjune251}, we know that $(\Omega^\bullet_C, \wedge_C, d_C)$ coincides with   the calculus $\dc$ and therefore, $\wedge_C \circ \nabla = d_C.$ 

Moreover, 
$$ \text{Ker} ( \wedge_C ) = \text{Ker} ( \wedge ) = \text{Ran} ( \psi_s )$$
by virtue of  Proposition \ref{9thjan251}.
Hence, part $(i)$ in Proposition \ref{7thjan251} allows us to conclude that $(1 - \psi_s) \nabla = d_{\psi_s}.$

Similarly, if $(1 - \psi_s) \nabla = d_{\psi_s},$ then from Proposition \ref{9thjan251}, (iii) of Lemma \ref{6thjan251} is satisfied and hence by the second assertion of Proposition \ref{7thjan251}, we have $\wedge \circ \nabla = d.$
\end{proof}

\subsection{Levi-Civita connections}

Till now, we have compared the notions of metric-compatibility and torsionlessness of connections used in \cite{BeggsMajid:Leabh} and \cite{meslandrennie1}. Combining these results, we can now easily compare the definitions of Levi-Civita connections in the two frameworks. In \cite[Definition 4.1]{meslandrennie1}, Mesland and Rennie consider the class of Hermitian differential structures. The following definition is a weakened version.

\begin{definition} \label{13thjan251}
A quintuple $( \Omega^1, d^\prime, \dagger, \psi, {}_B \langle \langle ~ , ~ \rangle \rangle  )$ is said to be a weak Hermitian differential structure on a $\ast$-algebra $B$ if
\begin{enumerate}

\item[$(i)$] $( \Omega^1, d^\prime, \psi )$ is a weak second order differential structure,

\item[$(ii)$] The triple $(\Omega^1, \dagger, {}_B \langle \langle ~ , ~ \rangle \rangle)$ is a weak $\dagger$-bimodule satisfying the condition $ d^\prime ( b^\ast ) = - d^\prime ( b )^\dagger $ for all $b \in B.$
\end{enumerate}
A left connection $\nabla^\prime$ is said to be a Levi-Civita connection for $( \Omega^1, d^\prime, \dagger, \psi, {}_B \langle \langle ~ , ~ \rangle \rangle  )$ if $\nabla^\prime$ is Hermitian with respect to the triplet  $(\Omega^1, \dagger, {}_B \langle \langle ~ , ~ \rangle \rangle)$ (i.e, \eqref{10thjan255} holds), and if $\nabla^\prime$ is torsionless for the weak second order differential structure $(\Omega^1, d^\prime, \psi)$ in the sense of Definition \ref{9thjune251}.
\end{definition}

Starting from a $\ast$-differential calculus and a real metric on the space of one forms, one has the following example of a weak Hermitian differential structure. 

\begin{example}\rm \label{14thjune251}
Let $\dc$ be a $\ast$-differential calculus on a $\ast$-algebra $B$ and $s: \Omega^2 \rightarrow \Omega^1 \otimes_B \Omega^1 $ be a $B$-bilinear map such that $(\Omega^\bullet, \wedge, d, s)$ is a quasi-tame differential calculus. Moreover, let $\metric$ be a real metric on $\Omega^1.$

Consider the first order differential calculus $(\Omega^1, \sqrt{-1} d).$   If $\psi_s$ denotes the map defined in \eqref{eq:psi_S}, then by Proposition \ref{9thjan251}, it follows that $(\Omega^1, \sqrt{-1} d, \psi_s)$ is a weak second order differential structure. Moreover, the Hermitian metric $\kH_g$ constructed from $\metric$ in Proposition \ref{27thdec242jb} yields a weak $\dagger$-bimodule $(\Omega^1, \dagger_{\star_{\Omega^1}}, {}_B \langle \langle ~ , ~ \rangle \rangle )$ by virtue of Corollary \ref{10thjan2510}. The equation $(\sqrt{-1} d) (b^*) = - \dagger_{\star_{\Omega^1}} ( (\sqrt{-1} d) (b) ) $ holds because of \eqref{18thjune251}.  In other words, $(\Omega^1, \sqrt{-1} d, \dagger_{\star_{\Omega^1}}, \psi_s, {}_B \langle \langle ~ , ~ \rangle \rangle  )$ is a weak Hermitian differential structure. 
\end{example}

In the sequel, we will refer to the tuple $(\Omega^1, \sqrt{-1} d, \dagger_{\star_{\Omega^1}}, \psi_s, {}_B \langle \langle ~ , ~ \rangle \rangle  )$ as the weak Hermitian differential structure associated to the data $(\dc, s, \metric).$

Now, we are ready for one of the main result in this article.

\begin{theorem} \label{thm:levicomp}
Suppose that $( \Omega^\bullet, \wedge, d, s  )$ is a quasi-tame $\ast$-differential calculus, $\metric$ is  a real metric on $\Omega^1$ and $(\nabla, \sigma)$ a $\ast$-preserving bimodule connection on $\Omega^1$ such that $\sigma$ is an isomorphism. Then the following statements are equivalent: 
\begin{enumerate}
\item[$(i)$] $(\nabla, \sigma )$ is a (unique) Levi-Civita connection for the metric $\metric$ on the $\ast$-differential calculus $\dc$ in the sense of Definition \ref{4thjuly241}.
\item[$(ii)$] $\sqrt{-1}\nabla$ is a (unique) Levi-Civita connection for the associated weak Hermitian differential structure  in the sense of Definition \ref{13thjan251}.
\end{enumerate}
\end{theorem}
\begin{proof}
Suppose that $(i)$ holds. Then, by Theorem \ref{9thmay252}, the equation \eqref{10thjan255} holds with $d^\prime = \sqrt{-1} d$ and $\nabla^\prime = \sqrt{-1} \nabla. $  Moreover, since   $\nabla$ is torsionless for the differential calculus $\dc,$   $\sqrt{-1} \nabla $ is a torsionless connection for the differential calculus $(\Omega^\bullet, \wedge, \sqrt{-1} d ).$ Therefore, by virtue of Theorem \ref{14thjan251}, we get that $(1 - \psi_s) ( \sqrt{-1}  \nabla ) = ( \sqrt{-1}  d)_{\psi_s}. $

The proof of the reverse implication and the uniqueness follows similarly.
\end{proof}

\section{On examples coming from cocycle deformations and Heckenberger-Kolb calculi} \label{21staugust254}

In this section, we prove that for cocycle deformations of a large class of Riemannian manifolds and the Heckenberger-Kolb calculi on quantized irreducible flag manifolds, there is a unique Levi-Civita connection in the sense of \cite{meslandrennie1} for every real covariant metric on the space of one-forms. This follows by applying Theorem \ref{thm:levicomp} to \cite[Proposition 9.28]{BeggsMajid:Leabh} and \cite[Theorem 6.14]{LeviCivitaHK}. Indeed, for both of these class of examples, we verify the existence of a covariant splitting of the $\wedge$ map and the $\ast$-preserving condition for the Levi-Civita connections obtained in \cite{BeggsMajid:Leabh,LeviCivitaHK}.

Let us begin by setting up the notations and conventions for Hopf algebras and their coactions. We refer to \cite{montgomery1993hopf} for more details on Hopf algebras and their coactions.

Let $A$ be a Hopf-algebra with coproduct $\Delta: A\to A\otimes A$, counit $\epsilon: A\to \mathbb{C}$ and antipode $S:A\to A$. The antipode is always assumed to be bijective and moreover, we will use Sweedler notation to write
\[
\Delta (a) = a_{(1)} \otimes a_{(2)}.
\]
For a Hopf algebra $A$ and a left $A$-comodule $V$, the coaction on an element
$v\in V$ is written in Sweedler notation as
$$\del{V}(v) = \mone{v}\ot\zero{v}.$$
An element $v \in V$ is called coinvariant if $\del{V}(v) = 1 \otimes v. $
We denote by $\qMod{A}{}{}{}$ the monoidal category of left $A$-comodules.

If $(B, \prescript{B}{}{\delta})$ is a left $A$-comodule algebra, then  $\qMod{A}{B}{}{B}$ will denote  the category of relative Hopf-modules. Thus,  an object $(\cE, \del{\cE})$ of $\cat$ is a $B$-bimodule and a left $A$-comodule such that
$$\del{\cE}(aeb)= \del{B}(a)\del{\cE}(e) \del{B}(c) \quad\text{for all } e \in \cE, a, b\in B .$$
A morphism $f: \cE \to \cF $  of the category $\cat$ is a $B$-bilinear map which is also $A$-covariant, i.e., 
$$\del{\cF} (f (e) )= (\id \ot f )\del{\cE} (e) $$
for all $e \in \cE.$ Similarly, the category $\qMod{A}{}{}{B}$ consists of right $B$-modules with compatible left $A$-coactions and the  morphisms are left $A$-covariant maps which are also right $B$-linear. 

Let us also recall that a Hopf algebra $(A, \Delta)$ is called a Hopf $\ast$-algebra if $A$ is a $\ast$-algebra and  $\Delta$ is a $\ast$-homomorphism. If $A$ is a Hopf $\ast$-algebra, a left $A$-comodule algebra $(B, \prescript{B}{}{\delta})$ is called a left $A$-comodule $\ast$-algebra if $B$ is a $\ast$-algebra and the coaction $\prescript{B}{}{\delta}$ satisfies the identity  
$$\prescript{B}{}{\delta} ( b^* ) = b^*_{(-1)} \otimes b^*_{(0)}. $$
In this case, a differential $*$-calculus $\dc$ on $B$ is said to be $A$-covariant if for all $k \geq 0,$ the bimodule $\form{k}$ is an  object of $\qMod{A}{B}{}{B},$ the map $\wedge$ is a morphism of $\cat$ and  $d: \Omega^k \rightarrow \Omega^{k + 1} $ is $A$-covariant. 

In the next two subsections, we will be dealing with covariant differential calculi and therefore, we introduce the following definition:
\begin{definition}
    An $A$-covariant  differential calculus $\dc$ is said to be covariantly quasi-tame if there exists a morphism $s: \Omega^2 \rightarrow \Omega^1 \otimes_B \Omega^1 $ in $\cat$ such that $\wedge \circ s = \id_{\Omega^2}.$ 
\end{definition}

Proposition \ref{rem:covsplit} gives a class of examples of covariantly quasi-tame differential calculi. In order to state the result, we quickly recall a few facts about quantum homogeneous spaces which will be needed in Subsection \ref{11thaugust251}. 

Let $A, H$ be Hopf-algebras and $\pi:A\to H$ be a surjective Hopf-algebra homomorphism. Consider the homogeneous right $H$-coaction on $A$  given by $\delta^A =(\id\otimes \pi)\circ \Delta$ and let 
$$B= A^{\text{co}(H)} := \{ a\in A: \delta^A (a) = a\otimes 1\}$$
be the coinvariant subalgebra.
We say that $B$ is a quantum homogeneous space if $A$ is faithfully flat as a right $B$-module.

Let $B^+$ denote $B \cap \ker (\epsilon). $ If $M$ is an object of $\qMod{A}{B}{}{B}$, then the map
$$  \frac{M}{B^+ M} \rightarrow H \otimes  \frac{M}{B^+ M}, ~  [ m ]  \mapsto \pi (m_{(- 1)}) \otimes [ m_{(0)} ] $$
is a coaction of $H$ on $ \frac{M}{B^+ M}$,
where $[ m ]$ denotes the equivalence class of $m \in M$ in $ \frac{M}{B^+ M}$. Thus, $ \frac{M}{B^+ M}$ is an object of the category  $\qMod{H}{}{}{B}$ with respect to the obvious right $B$-module structure  and we have a functor
\begin{equation} \label{2ndjuly241}
 \Phi: \qMod{A}{B}{}{B} \rightarrow \qMod{H}{}{}{B}, ~ M \rightarrow  \frac{M}{B^+ M}.
 \end{equation}

We define $ \modz{A}{B} $ to be the full subcategory of $ \qMod{A}{B}{}{B} $  whose objects $M$  are finitely generated as left $B$-modules and satisfy the condition $M B^+ = B^+ M$. It is easy to observe (see \cite[Remark 2.7]{LeviCivitaHK}) that any object of $ \modz{A}{B} $ is automatically projective as a left $B$-module.  Moreover, $\lmod{H}{}$ will denote the category of finite dimensional left $H$-comodules.

It is well-known (see \cite[Theorem 2]{Tak},  \cite[Section 6]{Skryabin2007}) that  the functor
\[
  \Phi: \modz{A}{B} \rightarrow \lmod{H}{}
\]
is a monoidal equivalence of categories. We will refer to $\Phi$ as Takeuchi's equivalence.

\begin{prop}  \label{rem:covsplit}
     Suppose that $B$ is a quantum homogeneous space of a Hopf algebra $A$ as above and $H$ is cosemisimple. If $\dc$ is an $A$-covariant differential calculus on $B$ such that $\Omega^k$ are objects of $ \modz{A}{B}, $ then $\dc$ is covariantly quasi-tame.
\end{prop}
\begin{proof}
The result was already observed in the proof of \cite[Proposition 4.1]{junaidbuachalla1}. However, we prove it for the sake of completeness.  

Since $ \Phi ( \Omega^1 \otimes_B \Omega^1 ) \xrightarrow{\Phi ( \wedge )} \Phi ( \Omega^2 )$ is a surjective linear map between finite dimensional vector spaces, there exists a linear map
$T: \Phi (\Omega^2) \rightarrow \Phi ( \Omega^1 \otimes_B \Omega^1 ) $ such that $\Phi ( \wedge ) \circ T = \text{id}_{\Phi ( \Omega^2 )}. $ Let $h$ denote the invariant integral of $H.$ Then it can be easily checked that the map
$$ T^\prime: \Phi (\Omega^2) \rightarrow \Phi ( \Omega^1 \otimes_B \Omega^1 ) ~ \text{defined by} ~ T^\prime ( v ) = h ( S ( v_{(-1)}  ) [T ( v_{(0)})]_{(-1)}  ) [T ( v_{(0)})]_{(0)}  $$
is a morphism in $\lmod{H}{}$ such that $ \Phi ( \wedge ) \circ T^\prime = \text{id}_{\Phi ( \Omega^2 )}. $

As $\Phi$ is an equivalence of categories, we have a morphism $s: \Omega^2 \rightarrow \Omega^1 \otimes_B \Omega^1  $ in $\modz{A}{B}$ with $\Phi ( s ) = T^\prime $ and $\wedge \circ s = \text{id}_{\Omega^2}.$
 \end{proof}

\subsection{The case of unitary cocycle deformations}

  We recall that if $A$ is a Hopf algebra, 
then a  convolution invertible map $\gamma:A \otimes A \ra \mathbb{C}$ is called a  \textbf{$2$-cocycle} if	
	\begin{equation*}
		\label{25thaug24} \gamma({{g}_{(1)}}\otimes {\one{h}}) \co{\two{g}\two{h}}{k} =  \co{\one{h}}{\one{k}} \co{g}{\two{h}\two{k}}   
	\end{equation*}
	for all $g,h, k \in A $ and   
	$\gamma({h\otimes1})= \epsilon(h) = \gamma({1\otimes h})$ for all  $ h\in A $. If $A$ is a Hopf $\ast$-algebra, then $\gamma$ is said to be unitary (see \cite[Definition 4.4]{SadeDeformSpecTrip}) if 
    $$\overline{\co{a}{b}}= \coin{S(a)^*}{S(b)^*}.$$

Throughout this subsection, $A$ will denote a Hopf $*$-algebra, $B$ an $A$-comodule $*$-algebra, $\gamma$ a  unitary 2-cocycle on $A$.

If $\gamma$ is a unitary $2$-cocycle on $A,$ then it is well-known (see \cite[Subsection 2.3]{SMFounds}) that there exists a Hopf $\ast$-algebra $(A_\gamma,\cdot_\gamma, \Delta, \epsilon_\gamma, S_\gamma, *_\gamma)$ isomorphic to $A$ as a coalgebra. We will denote this $\gamma$-deformed Hopf $\ast$-algebra by the symbol $A_\gamma.$  We refer to \cite[Subsection 2.3]{SMFounds} and \cite[Subsection 4.2]{SadeDeformSpecTrip} for the formulas of the structure maps.

In the presence of a unitary cocycle $\gamma,$ $A$-comodule $\ast$-algebras and objects of $\cat$ can also be deformed. Concretely, if $B$ is an $A$-comodule $\ast$-algebra, then we have an $A_\gamma$-comodule $\ast$-algebra $B_\gamma$ with multiplication and involution defined as 
 $$ a\cdot_\cot b = \co{\mone{a}}{\mone{b}} \zero{a}\zero{b}, ~ b^{*_\gamma}:= \bar{V}(\mone{b}^*) ~ \zero{b}^*.$$
Here, as in \cite[Subsection 4.2]{SadeDeformSpecTrip}, $\overline{V}: A \rightarrow \mathbb{C} $ is the map defined as 
\begin{equation} \label{12thaugust251}
\overline{V} (k) = \overline{\gamma} \big( k_{(2)} \otimes S^{-1} (k_{(1)})\big). 
\end{equation}
 
The cocycle deformation of an object $(M, \del{M})$  of $\cat$  is the pair  $( \Gamma ( M ), \del{M_\gamma})$ where $\Gamma ( M ) = M $ as a comodule equipped with $B_\gamma$-bimodule structures given by
\begin{align*}
	b\cdot_\cot m = \co{\mone{b}}{\mone{m}}\zero{b} \cdot \zero{m},	\quad m\cdot_\cot b =\co{\mone{m}}{\mone{b}} \zero{m}\cdot \zero{b}.
\end{align*}
 Then $\tcat$ is a monoidal category with the monoidal structure to be denoted by $\ot_{B_\gamma}$.
It is well-known that the canonical functor   
$	\Gamma: \cat \to \tcat $
 is a  monoidal equivalence.  
 The associated natural isomorphism  $\varphi$ between the functors $\ot_{B_\gamma} \circ(\Gamma\times \Gamma)$ and $\Gamma \circ \ot_B$  is given by
\begin{eqnarray*}\label{nt}
	\varphi_{V,W}: \Gamma(V) \ot_{B_\gamma} \Gamma(W) &\longrightarrow&  \Gamma(V \ot_B W),
	\\
	v \ot_{B_\gamma} w &\longmapsto & \co{\mone{v}}{\mone{w}}  \zero{v} \ot \zero{w}  , \nonumber
\end{eqnarray*}
for  objects $V,W$ in $\cat$. 
Let us also note that since $\varphi$ is a natural isomorphism, we have the equality
\begin{equation}\label{15thdec244}
	\left(\Gamma(f) \ot_{B_\cot} \Gamma(g)\right) \circ \varphi_{V_1,W_1}^{-1}= \varphi^{-1}_{V_2, W_2} \circ \Gamma(f \ot_B g)
\end{equation}
for morphisms $f \in \hm{}{V_1}{V_2},  g\in \hm{}{W_1}{W_2} $ in the category $\cat$.
Finally, for objects  $\cE_1, \cE_2, \cF_1, \cF_2$ in $\cat$ and morphisms $T: \cE_1 \ot_{B} \cE_2 \to \cF_1\ot_{B} \cF_2$ and $S: \cE_1 \ot_{B}\cE_2 \to B $,  define morphisms in $\tcat$
$$ T_\gamma: \Gamma(\cE_1)\ot_{B_\cot}\Gamma(\cE_2) \to \Gamma(\cF_1)\ot_{B_\cot}\Gamma(\cF_2)  \text{ and } S_\gamma: \Gamma(\cE_1)\ot_{B_\cot}\Gamma(\cE_2) \to B_\gamma $$
as
\begin{equation} \label{1stdec241jb}
T_\gamma:= \varphi_{\cF_1, \cF_2}^{-1}\circ \Gamma(T) \circ \varphi_{\cE_1, \cE_2} \text{ and } S_\gamma := \Gamma(S) \circ \varphi_{\cE_1, \cE_2}.
\end{equation}


Therefore, if $\dc$ is an $A$-covariant differential calculus on $B,$ the above discussion and  \eqref{1stdec241jb} allows us to make the following definitions:
$$ \form{k}_\gamma:= \Gamma(\form{k}), \form{\bullet}_\gamma:= \oplus_{k \geq 0} \form{k}_\gamma ~ \text{and} ~ \wedge_\gamma: \Omega^\bullet_\gamma \otimes_{B_\gamma} \Omega^\bullet_\gamma \rightarrow \Omega^\bullet_\gamma. $$
Moreover,  since $\Gamma$ is also a monoidal  equivalence between the categories $\qMod{A}{}{}{}$ and $\qMod{A_\gamma}{}{}{},$ we can define  
\begin{equation} \label{7thjuly251}
  d_\gamma:= \Gamma(d) ~ \text{and} ~   \nabla_\gamma= \varphi^{-1}_{\Omega^1, \cE}\circ \Gamma( \nabla)
\end{equation}
for an $A$-covariant connection $\nabla$ on an object $\cE$ in $\cat.$

The following result is well-known to the experts. We refer to \cite[Proposition 3.9]{2ndpaper} for a proof. 

\begin{prop}  \label{prop:5thdec241}
	If  $(\Omega^\bullet, \wedge, d)$ is an $A$-covariant differential $\ast$-calculus on $B,$ then the unitary 2-cocycle deformation 	$(\Omega^\bullet_\gamma, \wedge_\cot, d_\gamma)$ is an $A_\gamma$-covariant differential $*$-calculus on $B_\gamma$.
\end{prop}

Now, we introduce the $\ast$-differential calculus to which we want to apply Theorem \ref{thm:levicomp}. The results in Example \ref{15thdec24jb1} are well-known to the experts. However, for the sake of clarity, we borrow notations from \cite[Subsection 6.1]{2ndpaper} for the explanations. Let $ X ( \mathbb{R} )  $ denote the set of real points of a real smooth affine variety $X = \text{Spec} ( O_{\mathbb{R}} (X)  ).  $ We assume that $ X ( \mathbb{R} )  $ is non-empty. Then it is well-known that $ X ( \mathbb{R} ) $ is a smooth manifold and the $O_{\mathbb{R}} (X) $-module 
$$\Omega^1_{\mathbb{R}} (X ( \mathbb{R} )):= \text{Span} \{ b_0 db_1 : b_0, b_1 \in O_{\mathbb{R}} (X) \} $$
is finitely generated and projective. 
We define $ O_{\mathbb{C}} (X):= O_{\mathbb{R}} (X) \otimes_{\mathbb{R}} \mathbb{C} $ endowed with the involution 
$$ (f \otimes_{\mathbb{R}} \lambda )^\ast:= f \otimes_{\mathbb{R}} \overline{\lambda} $$
for $f \in  O_{\mathbb{R}} (X)$ and $\lambda \in \mathbb{C}. $ If $\wedge$ denotes the classical wedge map,  $d$ the de Rham differential and
$$ \Omega^k ( X (  \mathbb{R} )  ):= \text{Span} \{ b_0 db_1 \wedge \cdots \wedge db_k: b_i \in O_{\mathbb{C}} (X) \}, $$
then as observed in \cite[Subsection 6.1]{2ndpaper}, $(\Omega^{\bullet} ( X (  \mathbb{R} )  ), \wedge, d)$ is a $\ast$-differential calculus over the complex algebra  $O_{\mathbb{C}} (X).$

Likewise, if $ G ( \mathbb{R} )$ is the set of all real points of a  real smooth affine algebraic group $G = \text{Spec} ( O_{\mathbb{R}} (G)  )$ and  $O_{\mathbb{C}} (G):= O_{\mathbb{R}} (G) \otimes_{\mathbb{R}} \mathbb{C}, $ then  $O_{\mathbb{C}} (G)$ is a complex Hopf $\ast$-algebra. Furthermore, if there exists an action of varieties $G \times X \rightarrow X $ (see \cite[Subsection 3.1]{waterhouse}), then $(\Omega^\bullet (X (\mathbb{R}) ), \wedge, d)$ is an $O_{\mathbb{C}} (G) $-covariant $\ast$-differential calculus on $B:=O_{\mathbb{C}} (X).$

\begin{example} \label{15thdec24jb1}
Suppose that $G$ and $X$ are as above and   we have  an action of varieties $G \times X \rightarrow X.$ Let  $\widetilde{g_{\mathbb{R}}}: \Omega^1_{\mathbb{R}} ( X ( \mathbb{R} ) ) \otimes_{O_{\mathbb{R}} ( X )} \Omega^1_{\mathbb{R}} ( X ( \mathbb{R} ) ) \rightarrow O_{\mathbb{R}} ( X ) $ be a $G ( \mathbb{R} ) $-invariant Riemannian metric on $\Omega^1_{\mathbb{R}} ( X ( \mathbb{R} ) ).$ Then the following statements hold:
\begin{enumerate}
\item The map  $\widetilde{g_{\mathbb{R}}}$ gives rise to a canonical $O_{\mathbb{C}} (G) $-covariant real metric $\metric$ on $\Omega^1 ( X ( \mathbb{R} )  ) $ in the sense of Definition \ref{4thmay242}. 

\item The Levi-Civita connection $\nabla_{\mathbb{R}}$ on $\Omega^1_{\mathbb{R}} ( X ( \mathbb{R} ) ) $ for the metric   $\widetilde{g_{\mathbb{R}}} $ extends to the   unique  $O_{\mathbb{C}} ( G )$-covariant connection $\nabla$ on $\Omega^1 ( X ( \mathbb{R} ) )$ which is torsionless and compatible with the metric $(g, ( ~, ~ ) ).$
\end{enumerate}
Thus, if $\gamma$ is a unitary $2$-cocycle on the Hopf $\ast$-algebra $A:= O_{\mathbb{C}} ( G ), $ then by Proposition \ref{prop:5thdec241}, we have a deformed $A_\gamma$-covariant $\ast$-differential calculus, to be denoted by  $(\Omega^\bullet_\gamma, \wedge_\cot, d_\gamma).$
\end{example}
We refer to \cite[Subsection 6.1]{2ndpaper} for more details. 
 Now we start to verify the hypotheses of Theorem \ref{thm:levicomp} for the $\ast$-differential calculus  $(\Omega^\bullet_\gamma, \wedge_\cot, d_\gamma).$

\begin{prop}\label{prop:classtame}
The cocycle twisted calculus $\dctwisted$ of   Example \ref{15thdec24jb1} is covariantly  quasi-tame.
\end{prop}
\begin{proof}
In this proof we let $A:= O_\bbC(X)$ and $B:=O_\bbC(X)$.
Since $\Omega^1$ is finitely generated and projective as a left $B$-module and $\Gamma$ is an equivalence of categories, $\Omega^1_\gamma$ is finitely generated and projective as a left $B_\gamma$-module.

The canonical splitting of \eqref{4thjan251} for the differential calculus $\dc$ is given by the map 
$$ s ( \omega \wedge \eta ) = \frac{\omega \otimes_B \eta - \eta \otimes_B \omega}{2}. $$
Since $A$ is a commutative Hopf algebra, $s$ is covariant. Thus, the map $s_\gamma:= \varphi^{-1}_{\Omega^1, \Omega^1} \circ \Gamma ( s ) $ is also covariant. 

Since $\Gamma$ is an monoidal equivalence,
    it follows that 
    \begin{equation*}
        0 \to \Gamma(\mathrm{Ker}(\wedge)) \xrightarrow{\Gamma(i)} \Gamma\big( \form{1}\ot_B\form{1}\big) \xrightarrow{\Gamma(\wedge)}\form{2}_\gamma \to 0
    \end{equation*} is also split exact sequence with the splitting map $\Gamma(s).$ 
    Therefore, the covariant map $s_\gamma$ is a splitting of  
    $$0 \to \mathrm{Ker}(\wedge_\gamma)\xrightarrow{i}\form{1}_\gamma \otimes_{B_\gamma} \form{1}_\gamma\xrightarrow{\wedge_\gamma}\form{2}_\gamma\to 0 ~ \text{since} ~ \wedge_\gamma s_\gamma = \Gamma(\wedge) \varphi_{\Omega^1, \Omega^1}\varphi^{-1}_{\Omega^1, \Omega^1} \Gamma(s)= \id. $$
    This completes the proof. 
\end{proof}


If $B$ is a left $A$-comodule $\ast$-algebra, we know (\cite[Section 2.8]{BeggsMajid:Leabh}) that $\cat$ is a bar category and the same is true about $\tcat.$ In the sequel, $\Upsilon$ and $\Upsilon_\gamma$ will denote the natural isomorphisms of Definition \ref{15thjuly241} for the categories $\cat$ and $\tcat$ respectively.  We recall the following result from \cite{2ndpaper}. 

\begin{prop} (\cite[Theorem B.6]{2ndpaper})
Let $\gamma$ be a unitary $2$-cocycle on a Hopf $\ast$-algebra $A$ and $B$ an $A$-comodule $\ast$-algebra. 
For an object $\cE$ in $\cat,$ consider the map
\begin{equation*} \label{5thdec242jb}
\mathfrak{N}_\cE : \overline{\Gamma(\mathcal{E})} \to \Gamma(\overline{\mathcal{E}}),~ \fN_\cE(\overline{e})= \bar{V}(\mone{e}^*) \overline{\zero{e}},
\end{equation*}
where the map $\overline{V}$ is defined in \eqref{12thaugust251}. Then $ \mathfrak{N}: \text{bar} \circ \Gamma \rightarrow \Gamma \circ \text{bar} $ is a natural isomorphism.  Moreover, for objects $\cE, \cF$ in $\cat,$  the following equation holds:
\begin{equation} \label{8thjuly251}
 (\fN_{\cF} \otimes_{B_\gamma} \fN_{\cE}) \Upsilon_\gamma \overline{\varphi^{-1}_{\cE, \cF}} = \varphi^{-1}_{\overline{\cF}, \overline{\cE}} \Gamma ( \Upsilon ) \fN_{\cE \otimes_B \cF}. 
 \end{equation}
\end{prop}

In fact, the same conclusion can be drawn about the bar category $\qMod{A}{}{}{}.$

The following result has been proved in  \cite[Proposition 6.3]{beggsmajidtwisting} for the special case where the involution $\ast_\gamma$ on $A_\gamma$ coincides with $\ast.$ 

\begin{lem} \label{8thjuly253}
    Suppose that $\gamma$ is a unitary $2$-cocycle on $A$. Let $(\nabla, \sigma)$ be a $*$-preserving $A$-covariant bimodule connection on  $\Omega^1$ such that $\sigma$ is an $A$-covariant map. Then the connection $\nabla_\gamma$ defined in \eqref{7thjuly251} is also $*$-preserving.
\end{lem}
\begin{proof}
Since $\sigma$ is $A$-covariant, we can define the map $\sigma_\gamma$ by \eqref{1stdec241jb}. Then it is well-known that $(\nabla_\gamma, \sigma_\gamma)$ is a left bimodule connection.

We will use the notations $\star$ and $\star_\gamma$ to denote the morphisms $\Omega^1 \rightarrow \overline{\Omega^1} $ and $\Omega^1_\gamma \rightarrow \overline{\Omega^1_\gamma} $ defined in \eqref{14thmay251}. We recall from Definition \ref{15thapril252} that $\nabla$ is  $\ast$-preserving if 
    \begin{equation} \label{18thapril251}
        \ol{\nabla}\circ \star= \ol{\sigma} \circ \Upsilon^{-1}  \big( \star \ot_B \star \big) \circ   \nabla,
    \end{equation}
    where $\overline{\nabla}(\overline{e})= \overline{\nabla(e)}$.     
    By virtue of \cite[equation (B.6)]{2ndpaper}, we can write 
    \begin{align*}
        \ol{\nabla_\gamma}\circ \star_\gamma &= \ol{\varphi^{-1}_{\Omega^1, \Omega^1} \circ \Gamma(\nabla}) \circ \fN^{-1}_{\Omega^1} \circ \Gamma(\star)\\
        &= \overline{\varphi^{-1}_{\Omega^1, \Omega^1}} \circ \overline{\Gamma ( \nabla )} \circ \mathfrak{N}^{-1}_{\Omega^1} \circ \Gamma ( \star )\\
        &=  \overline{\varphi^{-1}_{\Omega^1, \Omega^1}} \mathfrak{N}^{-1}_{\Omega^1 \otimes_B \Omega^1} \Gamma ( \overline{\nabla} ) \circ \Gamma ( \star )    ~ \text{(by naturality of $\mathfrak{N}$)}\\
        &= \ol{\varphi^{-1}_{\Omega^1, \Omega^1}} \fN^{-1}_{\Omega^1 \ot_B \Omega^1} \Gamma(\ol{\nabla}\circ \star) \\
        &= \ol{\varphi^{-1}_{\Omega^1, \Omega^1}} \fN^{-1}_{\Omega^1 \ot_B \Omega^1} \Gamma(\ol{\sigma} \circ \Upsilon^{-1}  \big( \star \ot_B \star \big) \circ   \nabla) ~ \text{(by \eqref{18thapril251})} \\
        &=  \ol{\varphi^{-1}_{\Omega^1, \Omega^1}} \fN^{-1}_{\Omega^1 \ot_B \Omega^1} \Gamma(\ol{\sigma})  \circ \Gamma(\Upsilon^{-1}) \Gamma \big( \star \ot_B \star \big) \circ   \Gamma(\nabla)\\
        &=   \ol{\varphi^{-1}_{\Omega^1, \Omega^1}\Gamma(\sigma)} \fN^{-1}_{\Omega^1 \ot_B \Omega^1} \circ \Gamma(\Upsilon^{-1}) \Gamma \big( \star \ot_B \star \big) \circ   \Gamma(\nabla ) ~ \text{\big(by naturality of $\fN$\big)}\\
        &=  \ol{\sigma_\gamma} \ol{\varphi^{-1}_{\Omega^1, \Omega^1}} \fN^{-1}_{\Omega^1 \ot_B \Omega^1} \circ \Gamma(\Upsilon^{-1})  \varphi_{\ol{\Omega^1}, \ol{\Omega^1}}\big( \Gamma(\star) \ot_{B_\gamma} \Gamma(\star) \big) \varphi^{-1}_{\Omega^1, \Omega^1}\circ   \Gamma(\nabla ) ~\text{\big(by \eqref{1stdec241jb},\eqref{15thdec244}\big)}\\
        &= \ol{\sigma_\gamma} \Upsilon_\gamma^{-1} \big(\fN_{\Omega^1}^{-1} \ot_{B_\gamma} \fN_{\Omega^1}^{-1} \big) \big( \Gamma(\star) \ot_{B_\gamma} \Gamma(\star) \big) \nabla_\gamma ~ \text{\big(by \eqref{8thjuly251}\big)}\\
        &= \ol{\sigma_\gamma} \Upsilon_\gamma^{-1}  \big( \star_\gamma \ot_{B_\gamma} \star_\gamma \big) \nabla_\gamma 
    \end{align*}
 and we have applied \cite[equation (B.6)]{2ndpaper} again.   Hence, $\nabla_\gamma$ is $*$-preserving.
\end{proof}

Now we are in a position to state and prove the following result which extends the existence-uniqueness result in \cite[Theorem 6.12]{meslandrennie1} to arbitrary unitary cocycle deformations. We will be using the fact proved in Proposition \ref{prop:classtame} that $(\Omega^\bullet_\gamma, \wedge_\gamma, d_\gamma, s_\gamma)$ is a quasi-tame differential calculus. 

\begin{theorem} \label{9thjuly251}
   Let $\nabla$ be the covariant Levi-Civita connection for a real covariant metric $\metric$ on the space of one-forms of the $\ast$-differential calculus $\dc$ in Example \ref{15thdec24jb1}. Then $\ci \nabla_\gamma$ is the unique Levi-Civita connection for the weak Hermitian differential structure associated to the data $(\dctwisted, s_\gamma, \metrictwisted)$ as in Example \ref{14thjune251}.
\end{theorem}
\begin{proof}
   If $\sigma ( \omega \otimes_B \eta ) = \eta \otimes_B \sigma, $ then $(\nabla,\sigma)$ is the bimodule Levi-Civita connection for the metric $\metric$ and hence, by virtue of \cite[Proposition 9.28]{BeggsMajid:Leabh}, $(\nabla_\gamma, \sigma_\gamma)$ is the Levi-Civita connection for $\metrictwisted$ in the sense of Definition \ref{4thjuly241}.

    Since $\metric$ is a real metric on $\form{1},$  \cite[Proposition 4.8]{2ndpaper} implies that the pair  $ \metrictwisted$ defined by \eqref{1stdec241jb} is a real metric on $\form{1}_\gamma.$ By Proposition \ref{prop:classtame}, $\dctwisted$ is quasi-tame and hence, Example \ref{14thjune251} yields a weak Hermitian differential structure. 
 From Example \ref{15thdec24jb1}, we know that $\nabla$ is the extension of a connection on $\Omega^1_{\mathbb{R}} ( X ( \mathbb{R} ) ) .$ From this, it can be checked easily that $\nabla$ is $*$-preserving. Hence,  $\nabla_\gamma$ is also $*$-preserving by Lemma \ref{8thjuly253}. Therefore, the result follows from Theorem \ref{thm:levicomp}. 
\end{proof}

\subsection{The Heckenberger--Kolb Calculi}  \label{11thaugust251}

Let $G$ be a compact connected simple Lie group of rank $r$ with complexified Lie algebra $\mathfrak{g}.$ We will let $\{ \alpha_1, \cdots \alpha_r\}$ denote the simple roots while $\mathcal{P}^+$ will stand for the set of all dominant integral weight.

For $0 < q < 1,$ the symbol $U_q ( \mathfrak{g} ) $ will denote the Drinfeld-Jimbo quantized universal enveloping algebra of  \cite{DrinfeldICM,Jimbo1986}. 
Thus, it is a noncommutative associative algebra generated by symbols $E_i, F_i, K_i $ and $ K^{-1}_i$ for $i = 1, \cdots r,$ satisfying the relations (12)-(16) of \cite[Chapter 6]{KSLeabh}.

For a dominant integral weight $\mu$ of $\mathfrak{g}$ and maximal weight space $V_\mu$, define 
$$C(V_\mu):= 
 \mathrm{span}_{\mathbb{C}}\!\left\{ c_{f,v}: U_q (\mathfrak{g}) \to \bbC  \,| c^{{V}}_{f,v}(X) := f\big(X \triangleright v\big) \, v \in V_\mu,  f \in V_\mu^*\right\}.$$
 Then the function algebra $ \mathcal{O}_q ( G ) $ of the compact quantum group $G_q$ is defined as 
\begin{equation*}\label{eq:PeterWeyl}
\mathcal{O}_q(G) := \bigoplus_{\mu \in \mathcal{P}^+} C(V_{\mu}),
\end{equation*}
We note that $\mathcal{O}_q(G)$ is a cosemisimple Hopf algebra by construction. For more details, we refer to \cite{KSLeabh,NeshveyevTuset}.

Now  fix a subset $S $ of $\{ 1, 2, \cdots, r \}$ and consider the Hopf $*$-subalgebra
\begin{align*}
U_q(\mathfrak{l}_S) := \big< K_i, E_j, F_j \,|\, i = 1, \ldots, r; j \in S \big>.
\end{align*} 
The category of $U_q(\mathfrak{l}_S)$-modules is known to be semisimple. The Hopf $\ast$-algebra embedding $\iota_S:U_q(\mathfrak{l}_S) \hookrightarrow U_q(\mathfrak{g})$ induces the dual Hopf \mbox{$*$-algebra} map $\iota_S^{\circ}: U_q(\mathfrak{g})^{\circ} \to U_q(\mathfrak{l}_S)^{\circ}$. Since $\mathcal{O}_q(G) \subseteq U_q(\mathfrak{g})^{\circ}$ by construction, we can consider the restriction map
\begin{align*}
\pi_S:= \iota_S^{\circ}|_{\mathcal{O}_q(G)}: \mathcal{O}_q(G) \to U_q(\mathfrak{l}_S)^{\circ},
\end{align*}
and the Hopf $*$-subalgebra 
$
\mathcal{O}_q(L_S) := \pi_S\big(\mathcal{O}_q(G)\big) \subseteq U_q(\mathfrak{l}_S)^\circ.
$ Since the category of $U_q(\mathfrak{l}_S)$-modules is semisimple,  $\mathcal{O}_q(L_S)$ must be a cosemisimple Hopf algebra. 

Note that we have a surjective Hopf algebra homomorphism $\pi_S:\mathcal{O}_q ( G ) \rightarrow \mathcal{O}_q ( L_S ) $ and it turns out that $ \mathcal{O}_q ( G/L_S ):= \mathcal{O}_q \big(G\big)^{\mathrm{co}\left(\mathcal{O}_q(L_S)\right)} $ is a quantum homogeneous space. The algebra $ \mathcal{O}_q ( G/L_S )$ is called the algebra of functions on the quantum flag manifold associated to the set $S.$
\begin{definition}
A quantum flag manifold is {irreducible} if the defining subset of simple roots is of the form
$
S = \{1, \dots, r \} \setminus \{s\}
$
where $\alpha_s$ has coefficient $1$ in the expansion of the highest root of $\mathfrak{g}$.
\end{definition}

For irreducible quantum flag manifolds,  Heckenberger and Kolb  (\cite{HKdR}) constructed a canonical $\mathcal{O}_q ( G ) $-covariant differential calculi which we recall in the next result. The symbol $\modz{\cO_q(G)}{\cO_q(G/L_S)}$ stands for the monoidal category introduced at the beginning of this section and $\Phi$ denotes Takeuchi's equivalence.

\begin{theorem}\label{thm:HKClass}
Over any irreducible quantum flag manifold $\mathcal{O}_q(G/L_S)$, there exists a unique finite-dimensional left $\mathcal{O}_q(G)$-covariant differential $*$-calculus
\[
  \Omega^{\bullet}_q(G/L_S) \in \modz{\cO_q(G)}{\cO_q(G/L_S)}
\]
of classical dimension, i.e,
\begin{align*}
  \dim \Phi\!\left(\Omega^{k}_q(G/L_S)\right) = \binom{2M}{k}, & & \text{ for all $k = 0, \dots, 2 M$},
\end{align*}
where $M$ is the complex dimension of the corresponding classical manifold.
\end{theorem}

The calculus $( \Omega^{\bullet}_q(G/L_S) , \wedge, d)$ is known as the Heckenberger-Kolb calculus on $\mathcal{O}_q(G/L_S)$. The following result was proved in \cite{LeviCivitaHK}.

\begin{theorem}$\mathrm{\big(}$\cite[Theorem 6.14]{LeviCivitaHK}$\mathrm{\big)}$ \label{thm:Levi-CivitaHK}
Given an  $\mathcal{O}_q (G) $-covariant real metric $(g, (~, ~))$ on $\Omega^1_q(G/L_S)$, there exists a unique covariant bimodule connection $(\nabla, \sigma )$ on $\Omega^1_q (G/L_S) $  which is torsionless and compatible with the metric $ (g, (~, ~))$ in the sense of Definition \ref{4thjuly241}.
 \end{theorem}

 Now we are in a position to accommodate the Heckenberger-Kolb calculi in the framework of Mesland-Rennie. 

\begin{prop} \label{14thjuly252}
The Heckenberger Kolb calculi are covariantly quasi-tame and the connection $(\nabla, \sigma)$ of Theorem \ref{thm:Levi-CivitaHK} is the  unique $\ast$-preserving covariant connection on the space of one-forms.    
\end{prop}
\begin{proof}
We have already observed that $\mathcal{O}_q ( G/L_S ) $ is a quantum homogeneous space of  $\mathcal{O}_q ( G ) $ and  the Hopf algebra $\mathcal{O}_q ( L_S ) $ is cosemisimple. Moreover, $\Omega^k_q ( G/L_S ) $ are objects of the category $\modz{\cO_q(G)}{\cO_q(G/L_S)}.$ Therefore, Proposition \ref{rem:covsplit} implies that the Heckenberger-Kolb calculi are covariantly quasi-tame. 

Now we prove that $(\nabla, \sigma)$ is $\ast$-preserving. We will use the fact proved in \cite{ReAleJun} that $\sigma$ is invertible.  
Let $\dagger_{\Omega^1}: \Omega^1_q(G/L_S) \rightarrow \Omega^1_q(G/L_S) $ denote the map 
$\dagger_{\Omega^1} ( \omega ) = \omega^* $ as in \eqref{18thjune251} while
 $$\dagger_{\Omega^1 \otimes_{\mathcal{O}_q (G/L_S) } \Omega^1  } :  \form{1}_q(G/L_S) \otimes_{\mathcal{O}_q(G/L_S)} \form{1}_q(G/L_S) \rightarrow \form{1}_q(G/L_S) \otimes_{\mathcal{O}_q(G/L_S)} \form{1}_q(G/L_S) $$
 is defined as in \eqref{13thmay251}.
 Moreover, let us define
 $$ \nabla_1:= \dagger_{\Omega^1 \otimes_{\mathcal{O}_q (G/L_S) } \Omega^1  } \circ \nabla \circ \dagger_{\Omega^1} ~ \text{and} ~ \sigma_1:= \dagger_{\Omega^1 \otimes_{\mathcal{O}_q (G/L_S) } \Omega^1  } \circ \sigma \circ \dagger_{\Omega^1 \otimes_{\mathcal{O}_q (G/L_S) } \Omega^1  }. $$
 Then it can be checked that $(\nabla_1, \sigma_1)$  is an $\mathcal{O}_q(G)$-covariant right bimodule connection on $\Omega^1_q(G/L_S).$ Since $\sigma$ is invertible, so is $\sigma_1,$ and hence $\sigma_1^{-1} \nabla_1$ is an $\mathcal{O}_q(G)$-covariant left connection. Since $\Omega^1_q ( G/L_S ) $ admits a unique covariant left connection (\cite[Theorem 4.5]{BuarelativeHopfModule}) and $\nabla$ is a covariant connection, we conclude that  $\sigma_1^{-1} \nabla_1= \nabla$. 
    Consequently, $\nabla= \sigma \circ \dagger_{\Omega^1 \otimes_{\mathcal{O}_q (G/L_S) } \Omega^1  } \circ \nabla \circ \dagger_{\Omega^1}.$ But from \cite[page 266]{BeggsMajid:Leabh}, i.e. $\nabla$ is $\ast$-preserving. 
\end{proof}

Now the final result follows along the lines of the proof of Theorem \ref{9thjuly251}.

\begin{theorem} \label{14thjuly251}
Let $\metric$ be a  real covariant metric  on  $\Omega^1_q (G/L_S ) $ and $(\nabla, \sigma )$  the bimodule connection of Theorem \ref{thm:Levi-CivitaHK}.  Fix a covariant splitting $s$  of \eqref{4thjan251}.  

Then $\ci \nabla$ is the unique Levi-Civita connection (in the sense of Definition \ref{13thjan251}) for the weak Hermitian differential structure associated to the data $( ( \Omega^\bullet_q ( G/L_S ), \wedge, d ), s, \metric  ).$ 
\end{theorem}

We end this section by providing a second proof of Theorem \ref{14thjuly251} without using the result of \cite{ReAleJun} mentioned above. This requires the complex structure on the Heckeneberger-Kolb calculi. The following result follows from \cite{HK}, \cite{HKdR}, and~\cite{MarcoConj}.

\begin{prop} \label{prop:complexstructure}
For an irreducible quantum flag manifold  $\mathcal{O}_q(G/L_S)$, the Heckenberger-Kolb calculus  $(\Omega^{\bullet}_q(G/L_S), \wedge, d)$ admits precisely two left $\mathcal{O}_q(G)$-covariant complex structures, each of which is opposite to the other. 
\end{prop}

We will use the notation $\bN^2_0$ to denote the set $(\bN \cup \{0\}) \times (\bN \cup \{0\}).  $  Let us recall  (\cite{MMF2}) that an $\mathcal{O}_q ( G ) $-covariant  complex structure $\Omega^{(\bullet,\bullet)}_q ( G/L_S ) $ on the $\mathcal{O}_q ( G ) $-covariant differential calculus $ (\Omega^\bullet_q ( G/L_S ), \wedge, d) $ 
  is an $\bN^2_0  $-algebra grading $\bigoplus_{(a,b)\in \bN^2_0} \Omega^{(a,b)} _q ( G/L_S ) $
  for $\Omega^{\bullet} _q ( G/L_S ) $ such that, for all $(a,b) \in \bN^2_0$, $ \Omega^{(a,b)} _q ( G/L_S ) $ is a left  $\mathcal{O}_q ( G ) $-comodule and moreover,  we have
  \[
    \Omega^k _q ( G/L_S )  = \bigoplus_{a+b = k} \Omega^{(a,b)} _q ( G/L_S ) ,\qquad
    \big(\Omega^{(a,b)} _q ( G/L_S ) \big)^* = \Omega^{(b,a)} _q ( G/L_S ) ,\qquad
     \]
      \[
   d(  \Omega^{(a,b)}_q ( G/L_S ))  \subseteq \Omega^{(a+1,b)} _q ( G/L_S )  \oplus \Omega^{(a,b+1)} _q ( G/L_S ) .
  \]

Here, $\Omega^{(0,0)} _q ( G/L_S )  = \Omega^0 _q ( G/L_S )  = \mathcal{O}_q ( G/L_S ) $. As $ \Omega^\bullet_q (G/L_S)  = \bigoplus_{(a,b)\in \bN^2_0} \Omega^{(a,b)}_q (G/L_S) $ is an $\bN^2_0$-algebra grading, it follows that $\Omega^{(a,b)}_q (G/L_S)$ is an $\mathcal{O}_q ( G/L_S )$-bimodule.

We should mention that a combination of \cite[Lemma 2.15 and Remark 2.16]{MMF2}    implies that the above definition is equivalent to the definition of complex structures given by Beggs and Smith in \cite{BS} as well as that by Khalkhali, Landi and van Suijlekom in \cite{KLvSPodles}.

For each $a, b \in \bN_0$,  $\pi^{(a,b)}: \Omega^{a + b} _q ( G/L_S )   \rightarrow \Omega^{(a, b)} _q ( G/L_S )  $ will denote the canonical projections associated to the decomposition $  \Omega^k _q ( G/L_S )  = \bigoplus_{a+b = k} \Omega^{(a,b)} _q ( G/L_S ) $. Moreover, $\partial$ and $\overline{\partial}$ will denote the maps
$$ \partial:= \pi^{(a + 1, b)} \circ d: \Omega^{(a, b)} _q ( G/L_S )  \rightarrow \Omega^{(a + 1, b)} _q ( G/L_S )  ~ \text{and}$$
$$ \overline{\partial}:= \pi^{(a, b + 1)} \circ d: \Omega^{(a,b)} _q ( G/L_S )  \rightarrow \Omega^{(a, b + 1)} _q ( G/L_S ). $$
Since the calculus  $\Omega^{(\bullet,\bullet)}_q ( G/L_S ) $ is $\mathcal{O}_q ( G ) $-covariant, it follows that  the projections $\pi^{(a,b)}$ and $d$ are $A$-covariant and so the maps $\partial$ and $\overline{\partial}$ are also $A$-covariant.

The opposite complex structure on the Heckenberger-Kolb calculi  is given by the $\bN^2_0$-graded algebra
 $\bigoplus_{(a,b)\in \bN^2_0} \Omega^{(a,b), \text{op}}$, where
$$ \Omega^{(a,b), \text{op}}_q (G/L_S) := \Omega^{(b, a)}_q (G/L_S).  $$
Consequently, the $\partial$ and $\overline{\partial}$-operators of the opposite complex structure are given by:
$$ \partial_{\text{op}}:= \overline{\partial}: \Omega^{(a, b), \text{op}}_q (G/L_S) \rightarrow \Omega^{(a + 1, b), \text{op}}_q (G/L_S), ~  \overline{\partial}_{\text{op}}:= \partial: \Omega^{(a, b), \text{op}}_q (G/L_S) \rightarrow \Omega^{(a, b + 1), \text{op}}_q (G/L_S). $$
It is well-known (\cite[Proposition 3.11]{HKdR}) that the Heckenberger-Kolb calculi are factorizable, i.e,
the wedge maps from  $ \Omega^{(a,0)}_q ( G/L_S ) \otimes_{\mathcal{O}_q ( G/L_S )} \Omega^{(0,b)}_q ( G/L_S ) $ and $ \Omega^{(0,b)}_q ( G/L_S ) \otimes_{\mathcal{O}_q ( G/L_S )} \Omega^{(a,0)}_q ( G/L_S )  $ to $\Omega^{(a,b)}_q ( G/L_S )$
are isomorphisms with inverses $\theta^{(a,b)}_r$ and $\theta^{(a,b)}_l$ respectively.

As a consequence of factorizability, $\Omega^{(1,0)}_q ( G/L_S )$ and $\Omega^{(0,1)}_q ( G/L_S )$ are holomorphic bimodules (see  \cite[Definition 4.3]{BS}) with respect to the usual and opposite complex structure respectively. 

Now let $\mathscr{H}_1$ and $\mathscr{H}_2$ be  Hermitian metrics  on $\Omega^{(1,0)}_q ( G/L_S )$ and $\Omega^{(0,1)}_q ( G/L_S )$ respectively. By virtue of the existence of Chern connections \cite[Theorem 8.53 and Proposition 8.54]{BeggsMajid:Leabh},   there are unique connections $\nabla_{\mathrm{Ch}}$ on $\Omega^{(1,0)}_q(G/L_S)$ and   $\nabla_{\mathrm{Ch,op}}$ on $\Omega^{(0,1)}_q(G/L_S)$ such that $\nabla_{\mathrm{Ch}}$ is compatible with $\kH_1,$ $\nabla_{\mathrm{Ch,op}}$ is compatible with $\kH_2$ (in the sense of Definition \ref{19thoct231}) such that the following conditions are satisfied:
 $$ (\pi^{(0,1)} \otimes_{\mathcal{O}_q(G/L_S)} \id) \nabla_{\mathrm{Ch}}= \theta^{(1,1)}_l\circ \ol{\partial}, ~ (\pi^{(1,0)} \otimes_{\mathcal{O}_q(G/L_S)} \id) \nabla_{\mathrm{Ch,op}}= \theta^{(1,1)}_r\circ {\partial}.$$ 
We refer to \cite[Subsection 5.2 and Subsection 5.3]{LeviCivitaHK} for the details. 

\vspace{4mm}

{\bf Second proof of Theorem \ref{14thjuly251}:}

We already have the existence of a covariant $s: \Omega^2_q ( G/L_S ) \rightarrow \Omega^1_q ( G/L_S ) \otimes_{\mathcal{O}_q ( G/L_S )} \Omega^1_q ( G/L_S ) $ from Proposition \ref{14thjuly252}. The connection $\nabla$ of Theorem \ref{thm:Levi-CivitaHK} is torsionless in the sense of Definition \ref{9thjan252}. Therefore, by virtue of Theorem \ref{14thjan251}, the connection $\sqrt{-1}\nabla$ is torsionless for the weak second order differential structure $( \Omega^1_q (G/L_S), \sqrt{-1} d, \psi_s).$ Thus, we are left to show that $\sqrt{-1} \nabla$  is Hermitian  with respect to the weak $\dagger$-bimodule $(\Omega^1_q ( G/L_S ), \dagger_{\star_{\Omega^1}}, {}_{\mathcal{O}_q (G/L_S)}\langle\langle ~, ~ \rangle\rangle ).$

By \cite[Proposition 3.7]{LeviCivitaHK}, the  real covariant metric $\metric$ in the statement of Theorem \ref{14thjuly251}  satisfies the condition
\begin{equation*} \label{18thapril24jb1}
       (\omega, \eta)=0\text{ if }\omega, \eta \in \Omega^{(1,0)}_q(G/L_S) \text{ or if }\omega, \eta \in \Omega^{(0,1)}_q(G/L_S).
\end{equation*}
Let us define
    \begin{equation*}
            \mathscr{H}_1:\overline{\Omega^{(1,0)}_q ( G/L_S )} \to \prescript{}{\mathcal{O}_q ( G/L_S )}{ \Hom (\Omega^{(1,0)}_q ( G/L_S ), \mathcal{O}_q ( G/L_S ))} \quad \text{by} \quad \mathscr{H}_1(\overline{\omega})(\eta)= (\eta, \omega^*) ~ \text{and}
        \end{equation*}
        
        \begin{equation*}
           \mathscr{H}_2:\overline{\Omega^{(0,1)}_q ( G/L_S )} \to \prescript{}{\mathcal{O}_q ( G/L_S )}{ \Hom (\Omega^{(0,1)}, \mathcal{O}_q ( G/L_S ))} \quad \text{by} \quad \mathscr{H}_2(\overline{\omega})(\eta)=  (\eta, \omega^*).
        \end{equation*}
        Then it follows (see \cite[Theorem 4.4]{LeviCivitaHK}) that $\mathscr{H}_1$ and   $\mathscr{H}_2$ are $\mathcal{O}_q (G)$-covariant Hermitian metrics  on $ \Omega^{(1,0)}_q ( G/L_S ) $ and $\Omega^{(0,1)}_q ( G/L_S )$ respectively and we have $\mathscr{H}_g = \mathscr{H}_1 \oplus \mathscr{H}_2,$ where $\mathscr{H}_g$ is the Hermitian metric constructed from $\metric$ in  Proposition \ref{27thdec242jb}.

Then, as explained above, the Chern connections $\nabla_{\mathrm{Ch}} $ on $\Omega^{(1,0)}_q ( G/L_S )$ and $ \nabla_{\mathrm{Ch,op}}$ on on $\Omega^{(0,1)}_q ( G/L_S )$ are compatible with $\mathscr{H}_1$ and $\mathscr{H}_2$ respectively. From the proof of Theorem \ref{thm:Levi-CivitaHK}, we know that  $\nabla = \nabla_{\mathrm{Ch}} + \nabla_{\mathrm{Ch,op}}$ with respect to the decomposition $ \Omega^1_q ( G/L_S ) =  \Omega^{(1,0)}_q ( G/L_S ) \oplus \Omega^{(0,1)}_q ( G/L_S )$. Therefore, $\nabla$ is compatible with $\mathscr{H}_g.$ Thus, Lemma \ref{10thjan2512} implies that $\sqrt{-1} \nabla$ is Hermitian with respect to $(\Omega^1_q ( G/L_S ), \dagger_{\star_{\Omega^1}}, {}_{\mathcal{O}_q (G/L_S)}\langle\langle ~, ~ \rangle\rangle ).$
\qed

\appendix

\section{Spectral triples and \texorpdfstring{$*$}{}-differential calculi} \label{24thjune251}

In this appendix, we point out that the theory of spectral triple fit into the framework of $\ast$-differential calculi.
In the following definition, the symbol $B ( H )$ will stand for the set of {all} bounded linear operators on a Hilbert space $H.$

\begin{definition}[\cite{connes}]
A spectral triple on a unital $\ast$-algebra $B$ is a triple $( B, H, D ),$ where $H$ is a {separable} Hilbert space, $B$ is a unital $\ast$-subalgebra of $B ( H ) $  and $D$ is a (typically unbounded) self-adjoint operator on $H$ such that for all $b \in B,$ $[ D, \pi ( b ) ] $ extends to a bounded operator on $H.$    
\end{definition}

It is standard to assume that $(i+D)^{-1}$ is a compact operator but like in \cite{meslandrennie1}, we won't need this condition here.

Consider the first order differential calculus defined by the pair $( \Omega^1_C ( B ), d_C ), $ where 
$$\Omega^1_C ( B ) = \text{span} \big\{ [ D,  b  ] c : b, c \in B \big\}  \subseteq B ( H ) ~ \text{and} ~ d_C ( b ) = \sqrt{-1} [ D,  b ]. $$  
By construction, $\Omega^1_C ( B )$ is closed under the usual involution $\ast_{B ( H )}$ on $B ( H ) $ and due to the definition of $d_C,$  the pair $( \Omega^1_C ( B ), d_C ), $ is a first order $\ast$-differential calculus.  

Now we can apply the recipe of Definition \ref{20thmay251} to the triple $(\Omega^1_C ( B ), d_C, B ( H ) )$ to obtain Connes' differential calculus $(\Omega^\bullet_C, \wedge_C, d_C)$ for the spectral triple. We refer to \cite[Section 7.2]{landibook} for  exposition.  

It is well-known and easy to check  that $J^\ast_d:= \oplus_{k \geq 0} J^k_d$ is a two-sided graded differential $\ast$-ideal of $\oplus_{k \geq 0} m_k ( \Omega^k_u ) $ and hence, each $\Omega^k_C ( B ) $ is closed under the antilinear involution $\ast$ defined by
$$ ( T + J^k_d   )^\ast :=  T^{\ast_{B ( H )}} + J^k_d ~ \text{for all} ~ T \in m_k ( \Omega^k_u ( B ) ),  $$
where $\ast_{B(H )}$ is the canonical involution on $B ( H ) $ and the symbol $m_k$ is as  defined in Section \ref{19thjune251}. Unfortunately, the conditions in \eqref{eq:starcal} may fail to hold for the Connes' differential calculus. For example, if $b_1, b_2 \in B,$ then
\begin{eqnarray*}
( d_C ( b_1 ) \wedge_C d_C ( b_2 )   )^\ast &=& (  [\sqrt{-1} D, b_1  ] [\sqrt{-1} D, b_2  ]  )^{\ast_{B(H)}} + J^2_d\\
&=&   [\sqrt{-1} D,  b^\ast_2  ] [\sqrt{-1} D,  b^\ast_1  ] + J^2_d\\
&=& d_C ( b^\ast_2 ) \wedge_C d_C ( b^\ast_1 )
\end{eqnarray*}
which may not be equal to $- d_C ( b^\ast_2 ) \wedge_C d_C ( b^\ast_1 )$ (equal if and only if it is $0$). 

However, this problem can be rectified by changing the involution on the higher forms.

\begin{lem} \label{thm:star_calculus}
    Suppose {that} $\dc $ is a differential calculus on a $\ast$-algebra $B$. Also assume  that there exists a graded antilinear involution $\ast$ on $\Omega^\bullet$ such that for all $b,b_0, b_1,  \cdots b_{k+1}\in B$, 
    $$d(b^*)=(db)^* \mathrm{~and~} \big(b_0 d(b_1) \wedge \cdots \wedge d(b_k) b_{k + 1} \big)^*= b^*_{k + 1} d(b_k^*) \wedge d(b_{k-1}^*)\cdots \wedge d(b_1^*) b_0^*.$$    
    Consider the graded antilinear involution $*_1$ on  $\form{\bullet}$ defined by 
    \begin{equation}\label{new_inv_bg}
    *_1(b_0 db_1 \wedge \cdots \wedge db_k b_{k+1}) := \begin{cases}
    b_{k+1}^*db_k^* \wedge db_{k-1}^* \wedge \cdots \wedge db_1^* b_0^* & \mathrm{for ~} k\equiv 0,1 ~\mathrm{mod} ~4 \\
        - b_{k+1}^* db_k^* \wedge db_{k-1}^* \wedge \cdots \wedge db_1^* b_0^* & \mathrm{for ~} k\equiv 2,3~ \mathrm{mod} ~4 
    \end{cases}.
    \end{equation}
    Then, $\dc$ is a $\ast$-differential calculus with respect to the involution $\ast_1.$ 
\end{lem}
\begin{proof}
    By the above discussion, we are only to required to check the validity of the equation \eqref{eq:starcal}. {Consider} $\omega = b_0 db_1 \wedge \cdots \wedge db_k \in \Omega^k. $ {To prove} the equation $ d ( \omega^{\ast_1} ) = d ( \omega )^{\ast_1}, $ we verify it for the cases $k = 0,1,2,3 $ (mod $4$) individually. For example, if $k = 2$ (mod $4$), then
\begin{eqnarray*}
 d ( \omega^{\ast_1}  ) &=& d ( - db^*_k \wedge \cdots \wedge d ( b^*_1 ) b^*_0  ) = - ( - 1 )^k db^*_k \wedge \cdots \wedge d ( b^*_1 ) \wedge d ( b^*_0 )\\
 &=& - db^*_k \wedge \cdots \wedge d ( b^*_1 ) \wedge d ( b^*_0 ) = ( d b_0 \wedge \cdots \wedge d ( b_k )  )^{\ast_1} = ( d \omega )^{\ast_1}. 
\end{eqnarray*}
 %
%
%
Next, if 
$ \omega = b_0 db_1 \wedge \cdots \wedge db_k \in \form{k}  ~ \text{and} ~ \eta = da_1 \wedge \wedge da_l a_{l + 1} \in \form{l},$
then the equation $ (\omega \wedge \eta )^{\ast_1} = (- 1)^{kl} \eta^{\ast_1} \wedge \omega^{\ast_1} $ can be proved by checking it for the {sixteen} cases $k,l = 0,1,2,3 $ (mod $4$). 
We illustrate this by taking $k = 2 $ (mod $4$) and $l = 2$ (mod $4$). In this case, for $\omega = d(a_1) \wedge d(a_2) \wedge \cdots\wedge d(a_{l})a_{l+1}$  and $\eta \in b_0 \wedge db_1\wedge\cdots  db_l,$ 
\begin{eqnarray*}
 ( \omega \wedge \eta )^{\ast_1} &=& a^*_{l + 1} d ( a^*_l ) \wedge \cdots \wedge d ( a^*_1 ) \wedge d ( b^*_k ) \wedge \cdots \wedge d ( b^*_1 ) b^*_0\\
 &=& ( - 1 ) ( -  a^*_{l + 1} d ( a^*_l ) \wedge \cdots \wedge d ( a^*_1 ) ) \wedge ( -  d ( b^*_k ) \wedge \cdots \wedge d ( b^*_1 ) b^*_0)\\
 &=& ( - 1 )^{kl} \eta^{\ast_1} \wedge \omega^{\ast_1}.
\end{eqnarray*}
This proves the theorem.
\end{proof}


Since the Connes' calculus $(\Omega^\bullet_C, \wedge_C, d_C)$ for a spectral triple satisfies the conditions of  Lemma \ref{thm:star_calculus}, we have the following corollary:

\begin{cor}
If $(B, H, D)$ is a spectral triple and $\Omega^\bullet_C$ is equipped with the $\ast$-structure as in \eqref{new_inv_bg}, then $( \Omega^\bullet_C, \wedge_C, d_C  )$  is a $\ast$-differential calculus.  
\end{cor}







\bibliographystyle{alpha} 
\bibliography{references}

\end{document}